\documentclass[11pt,oneside]{amsart}

\usepackage[T1]{fontenc}
\usepackage[utf8]{inputenc}

\allowdisplaybreaks

\usepackage{vmargin}
\setmarginsrb{3cm}{2cm}{3cm}{2cm}{0cm}{2cm}{0cm}{1cm}

\usepackage{amssymb,amsmath}

\usepackage{tikz}
\usepackage{todonotes}
\usetikzlibrary{automata,arrows.meta}
\definecolor{MyGreen}{rgb}{0.13,0.55,0.13}

\usepackage{graphicx}

\theoremstyle{plain}
\newtheorem{theorem}{Theorem}
\newtheorem{lemma}[theorem]{Lemma}
\newtheorem{corollary}[theorem]{Corollary}
\newtheorem{proposition}[theorem]{Proposition}
\theoremstyle{definition}
\newtheorem{definition}[theorem]{Definition}
\newtheorem{example}[theorem]{Example}

\numberwithin{equation}{section}

\newcommand{\R}{\mathbb R}
\newcommand{\N}{\mathbb N}
\newcommand{\Z}{\mathbb Z}
\newcommand{\A}{\mathcal{A}}

\newcommand{\Beta}{B}
\newcommand{\val}{\mathrm{val}}
\newcommand{\lex}{\mathrm{lex}}
\newcommand{\rad}{\mathrm{rad}}
\newcommand{\ceil}[1]{\left\lceil#1\right\rceil}
\newcommand{\floor}[1]{\left\lfloor#1\right\rfloor}
\newcommand{\DB}{d_B}

\makeatletter
\newcommand*\bigcdot{\mathpalette\bigcdot@{.5}}
\newcommand*\bigcdot@[2]{\mathbin{\vcenter{\hbox{\scalebox{#2}{$\m@th#1\bullet$}}}}}
\makeatother

\title{Substitutions and Cantor real numeration systems}
\author{\'Emilie Charlier$^{1,*}$, Célia Cisternino$^{1}$, Zuzana Mas\'akov\'a$^2$ and Edita Pelantov\'a$^2$}
\address{$^1$Department of Mathematics\\
University of Li\`ege\\
All\'ee de la D\'ecouverte 12,
4000 Li\`ege, Belgium\\
$^2$Department of Mathematics FNSPE
\\
Czech Technical University in Prague\\
Trojanova 13, 120 00 Praha 2, Czech Republic
}
\thanks{\emph{E-mail address:} \texttt{echarlier@uliege.be, celiacisternino@gmail.com, zuzana.masakova@fjfi.cvut.cz  and edita.pelantova@fjfi.cvut.cz}\\
$^*$Corresponding author.}

\begin{document}
\tikzset{elliptic state/.style={draw,ellipse,minimum width=6cm,minimum height=1.5cm}}

\begin{abstract}
We consider Cantor real numeration system as a frame in which every non-negative real number has a positional representation. The system is defined using a bi-infinite sequence  $\Beta=(\beta_n)_{n\in\Z}$ of real numbers greater than one. We introduce the set of $\Beta$-integers  and code the sequence of gaps between consecutive $\Beta$-integers by a symbolic sequence in general over the alphabet $\N$. We show that this sequence is $S$-adic. We focus on alternate base systems, where the sequence $\Beta$ of bases is periodic and characterize alternate bases $\Beta$, in which $\Beta$-integers can be coded using a symbolic sequence $v_{\Beta}$ over a finite alphabet. With these so-called Parry alternate bases we associate some substitutions and show that $v_\Beta$ is a fixed point of their composition. The paper generalizes results of Fabre and Burdík et al.\ obtained for the R\'enyi numerations systems, i.e., in the case when the Cantor base $\Beta$ is a constant sequence.
\end{abstract}

\maketitle

\bigskip
\hrule
\bigskip



\bigskip
\hrule
\bigskip

\section{Introduction}
\label{Sec : Introduction}

In the pursuit of a generalization of the famous Cobham's theorem, Fabre~\cite{Fabre1995} has associated  substitutions with numeration systems with one real base $\beta$ which is a Parry number. Such substitutions have for their fixed points infinite sequences which are sometimes called Parry words, or $k$-bonacci-like words. They include the famous Fibonacci sequence corresponding to the golden ratio numeration system.
Although Fabre himself viewed the substitutions in the context of automata recognizing the language of an associated linear numeration system, it can be derived already from Thurston~\cite{Thurston1989} that the substitutions may serve for generating the set of numbers with integer expansion in base $\beta$, the so-called $\beta$-integers. These numbers were identified as a suitable tool in the description of mathematical models of quasicrystals~\cite{BurdikFrougnyGazeauKrejcar1998}, since their arithmetic and diffractive properties generalize the arithmetic and diffractive properties of ordinary integers which give an underlying structure of periodic crystals.
Certain physical properties of materials such as electron conductivity can be derived from the character of the spectrum of a Schroedinger operator associated with the model. As identified in the works of Damanik~\cite{Damanik2000,DamanikGhezRaymond2001}, the spectrum depends on  combinatorial properties of infinite sequences coding distances in the $\beta$-integers, such as the factor complexity, palindromicity, repetition index, etc. Other physical  properties can be derived from their geometric behavior~\cite{Alietal2009,BalkovaGazeauPelantova2008,GuimondMasakovaPelantova2004}.

First studies of combinatorics of Parry sequences
date back to the paper of Burdík et al.~\cite{BurdikFrougnyGazeauKrejcar1998}. More general results about their factor complexity can be found in~\cite{Frougny&Masakova&Pelantova:2004,BernatMasakovaPelantova2007}. Palindromic complexity is the subject of~\cite{AmbrozFrougnyMasakovaPelantova2006}.
Turek in~\cite{Turek2015} gives a tool to compute abelian complexity and balances of Parry words.
The balance properties have been  used  for determining  arithmetical features of numeration systems~\cite{BalkovaPelantovaTurek2007}.
Recently, Gheeraert et al.~\cite{GheeraertRomanaStipulanti2023} study their string attractors.

Our aim is to generalize the notion of $\beta$-integers and corresponding symbolic sequences to positional numeration systems defined using multiple bases.
The so-called Cantor real numeration system was introduced
recently in~\cite{CharlierCisternino2021} for representations of numbers in the interval $[0,1)$.
The system is given by a sequence $\Beta = (\beta_n)_{n\geq 1}$ of real bases $\beta_n>1$.  $\Beta$-expansion of a number  $x \in [0,1)$ is the lexicographically greatest string $(x_n)_{n\geq 1}$ of non-negative integers such that  $x=\sum_{n=1}^{+\infty} \frac{x_n}{\beta_1\beta_2 \cdots \beta_n}$. The R\'enyi numeration system can be viewed as the special case when the sequence $\Beta$ is constant.  The adjective "Cantor real" added to the name of numeration system is linked to the fact that the idea of using multiple bases comes already from Cantor~\cite{Cantor1869}, who however, required all the bases $\beta_n$ to be integers.
When the Cantor real base $\Beta$ is a periodic sequence, we speak of alternate base systems.

In our paper, we first modify the definition of a Cantor real system so that it represents also positive numbers outside of the unit interval, see Section~\ref{Sec : BRepresentations}. For that, we consider a bi-infinite sequence $\Beta=(\beta_n)_{n\in\Z}$.
Consequently we introduce
 the set of $\Beta$-integers (Section~\ref{Sec : B-integers}) and determine the distances between its consecutive elements (Section~\ref{Sec : quasi-greedy}).
 For the special case of  alternate base systems, we study in Section~\ref{Sec : U-integers in alternate bases} when the distances between consecutive $\Beta$-integers take finitely many values.
 It turns out that this happens precisely when the base is an Parry alternate base, as defined in~\cite{Charlier&Cisternino&Masakova&Pelantova:2022}.
The sequence of distances between $\Beta$-integers then can be coded by an infinite symbolic sequence $v_\Beta$ over a finite alphabet. Then we show that $v_\Beta$ is a fixed point of a primitive substitution. In Section~\ref{sec:sturmian} we describe the bases $\Beta$ which yield a sturmian infinite sequence $v_\Beta$.

In order to give intuition all along the text, we refer the reader to Section~\ref{Sec : Example}, where we use a running example in order to illustrate the concepts under study and the obtained results.

\section{Combinatorics on words}
\label{Sec : Preliminaries}

An \emph{alphabet} is a countable set. In this paper we consider both finite and infinite alphabets, usually $\A=\{0,1,\dots,k-1\}$ for $k\in\N$, or $\A=\N$. The set of all finite words over the alphabet $\A$, equipped with the empty word $\epsilon$ and the operation of concatenation, is a monoid denoted by $\A^*$. The \emph{length} of a word $w=a_0a_1\cdots a_{n-1}\in\A^*$ with $a_i\in\A$ is denoted $|w|=n$. We have $|\epsilon|=0$. For $a\in\A$ and $w=a_0\cdots a_{n-1}\in\A^*$ we denote the number of occurrences of $a$ in $w$ by $|w|_a$, i.e., $|w|_a = \#\{ 0\leq i <n : w_i=a\}$.

We consider right-sided infinite words ${{v}}=a_0a_1a_2\cdots\in\A^{\N}$ with $a_i\in\A$.
Such an infinite word is \emph{purely periodic} if it can be written as an infinite concatenation of a single finite word $w\in\A^*$. We write ${v}=www\cdots =w^\omega$. The infinite word ${v}$ is \emph{eventually periodic} if
there exist words $u,w\in\A^*$ such that ${v}=uw^\omega$.

The \emph{frequency} of a letter $a$ in an infinite word ${v}=a_0a_1a_2\cdots$ is defined as the limit
$$
\rho_a=\lim_{n\to\infty} \frac{\#\{0\leq i<n :a_i=a\}}{n},
$$
if it exists.

Given an alphabet $\A$, a \emph{substitution} $\psi$ over $\A$ is a monoid endomorphism $\psi:\A^*\to\A^*$ such that there is a letter $a\in\A$ with $\psi(a)=aw$ for $w\in\A^*$, $w\neq \epsilon$. The action of a substitution can be extended to right-sided infinite sequences by
\[
	\psi(a_0a_1a_2\cdots) = \psi(a_0)\psi(a_1)\psi(a_2)\cdots.
\]
An infinite sequence $v\in\A^{\N}$ is a \emph{fixed point} of a substitution $\psi$ if $\psi(v)=v$. It is easy to see that a substitution $\psi$ has always a fixed point, namely $\lim_{n\to\infty}\psi^n(a)$, where the limit is considered over the product topology.

The \emph{incidence matrix} of a substitution $\psi$ over the alphabet $\A=\{0,1,\dots,k-1\}$ with $\#\A=k$, is a $k\times k$ matrix $M_\psi$ with non-negative entries defined as $(M_\psi)_{i,j}=|\psi(j)|_i$. For two substitutions $\psi,\tilde{\psi}$ over $\A$ it holds that $M_{\psi\circ\tilde{\psi}} = M_\psi M_{\tilde{\psi}}$.

An infinite sequence $v\in\A^{\N}$ is an \emph{$S$-adic sequence} if there exists a sequence $(\sigma_n)_{n\in\N}$ of substitutions over an alphabet $\A$ and a letter $a\in\A$ such that
\[
	v=\lim_{n\to\infty} \big(\sigma_0\circ \sigma_1 \circ \cdots \circ \sigma_{n-1}\circ \sigma_n\big) (a).
\]
For a more general definition of $S$-adic systems see~\cite{BertheDelecroix2014}.
A substitution $\psi$ is said to be \emph{primitive} if $M_{\psi}^n$ is a positive matrix for some $n\geq 1$.
It can be shown that a fixed point of a primitive substitution has well-defined non-zero frequencies of all letters. Moreover, the positive vector $(\rho_0,\dots,\rho_{k-1})^{\top}$ is an eigenvector of $M_{\psi}$ corresponding to the Perron-Frobenius eigenvalue of $M_\psi$. See~\cite{Queffelec:2010} for more details.

\section{Two-way real Cantor bases}
\label{Sec : BRepresentations}

In this section, we generalize the framework of~\cite{CharlierCisternino2021} in order to be able to represent all real numbers $x\ge 0$. Let $\Beta=(\beta_n)_{n\in\Z}$ a bi-infinite sequence of real numbers greater than one and such that $\prod_{n=0}^{+\infty}\beta_{n}=+\infty$ and $\prod_{n=1}^{+\infty}\beta_{-n}=+\infty$. We call such a sequence a \emph{two-way real Cantor base}. The associated value map is defined only on those sequences $a=(a_n)_{n\in\Z}$ in $\N^{\Z}$ having a left tail of zeros, that is, such that there exists some $N\in\Z$ such that $a_n=0$ for all $n\ge N$.  For such a sequence $a$, we define
\[
	\val_{\Beta}(a)=\sum_{n=0}^N a_n\beta_{n-1}\cdots\beta_0+\sum_{n=1}^{\infty}\frac{a_{-n}}{\beta_{-1}\cdots\beta_{-n}},
\]
provided that the infinite sum defines a convergent series. In this case, we say that $a$ is a \emph{$\Beta$-representation} of the so-obtained non-negative real number $\val_{\Beta}(a)$. The sequence $a$ is usually written as
\[
	\begin{cases}
	a_N\cdots a_0\bigcdot a_{-1}a_{-2}\cdots 	& \text{if }N\ge 0;	\\
	\hspace{1.25cm}0\bigcdot 0^{-N-1}a_N a_{N-1}\cdots		& \text{if }N< 0.	
	\end{cases}
\]

In general, a non-negative real number $x$ can have more than one $\Beta$-representation. Among all of these, we consider the greedy one defined as follows.

First, consider the case where $x\in[0,1)$. We use the greedy algorithm defined in the one-way Cantor base $(\beta_{-1},\beta_{-2},\ldots)$ as defined in~\cite{CharlierCisternino2021}: first, set $r_{-1}=x$, and then, for $n\le -1$, iteratively compute $a_{n}=\floor{\beta_{n} r_{n}}$ and $r_{n-1}=\beta_{n} r_{n}-a_{n}$. In the two-way real Cantor base $\Beta$, the \emph{$\Beta$-expansion} of $x$ in then denoted by
\[
	 \DB(x)=0\bigcdot a_{-1}a_{-2}\cdots.
\]
Since we have assumed that $\prod_{n=1}^{+\infty}\beta_{-n}=+\infty$, this greedy algorithm converges, meaning that $\val_{\Beta}(\DB(x))=x$.

Next, we consider the case where $x\ge 1$. In this case, we will need to use the bases $\beta_n$ for $n\ge 0$ as well. Let $N\ge 0$ be the minimal integer such that $x<\beta_N\cdots \beta_0$. Note that such an integer exists since we have assumed that $\prod_{n=0}^{+\infty}\beta_{n}=+\infty$. We compute the greedy $\Beta^{(N)}$-representation of $\frac{x}{\beta_N\cdots \beta_0}$, where $\Beta^{(N)}$ is the shifted one-way Cantor base $(\beta_N,\beta_{N-1},\ldots)$. If the obtained expansion is $a_N a_{N-1}\cdots$, then the \emph{$\Beta$-expansion} of $x$ is denoted by
\[
	\DB(x)=a_N\cdots a_0\bigcdot a_{-1}a_{-2}\cdots.
\]	
In particular, the obtained digits $a_n$ belong to the alphabet $\{0, \ldots, \ceil{\beta_n}-1\}$ for all $n$. Also note that, in this setting, the \emph{$\Beta$-expansion} of $1$ is always $\DB(1)=1\bigcdot 0^\omega$.

Let us collect some properties of the
$\Beta$-expansions. These properties are straightforward consequences of analogue results in~\cite{CharlierCisternino2021} for one-way Cantor real bases.

\begin{lemma}
\label{Lem : GreedyCondition}
Let $\Beta=(\beta_n)_{n\in \Z}$ be a two-way real Cantor base. A  sequence $(a_n)_{n\in\Z}$ over $\N$ having a left tail of zeros is the $\Beta$-expansion of some non-negative real number if and only if we have
\[
	\sum_{n=0}^N a_n\beta_{n-1}\cdots\beta_0+\sum_{n=1}^{\infty}\frac{a_{-n}}{\beta_{-1}\cdots\beta_{-n}} < \beta_N\cdots\beta_0\quad\text{for all }N\ge 0
\]
and
\[
	\sum_{n=N}^{\infty}\frac{a_{-n}}{\beta_{-1}\cdots\beta_{-n}} < \frac{1}{\beta_{-1}\cdots\beta_{-N+1}}\quad\text{for all }N>0.
\]
\end{lemma}

Let us define a total order on the $\Beta$-expansions of non-negative real numbers. Let $a=(a_n)_{n\in\Z}$ and $b=(b_n)_{n\in\Z}$ be sequences in $\N^{\Z}$ different from $(0)_{n\in\Z}$, for which there exists some minimal indices $M,N\in\Z$ such that $a_n=0$ for all $n>M$ and $b_n=0$ for all $n>N$. Then we say that $a$ is less than or equal to $b$ in the \emph{radix order} if either $M<N$, or if $M=N$ and $(a_{M-n})_{n\in\N}$ is lexicographically less than or equal to $(b_{M-n})_{n\in\N}$. In this case, we write $a\le_{\rad} b$. We also set $(0)_{n\in\Z}\le_{\rad} a$ for any sequence $a\in\N^{\Z}$ having a left tail of zeros. As usual, we write $a<_{\rad} b$ if $a\le_{\rad} b$ and $a\ne b$.

\begin{lemma}
\label{Lem : RadixOrder}
Let $\Beta=(\beta_n)_{n\in \Z}$ be a two-way real Cantor base.
\begin{enumerate}
\item For all non-negative real numbers $x$ and $y$, we have $x<y$ if and only if $\DB(x)<_{\rad} \DB(y)$.
\item The $\Beta$-expansion of a non-negative real number is maximal in the radix order among all its $\Beta$-representations.
\end{enumerate}
\end{lemma}

\section{$\Beta$-integers}
\label{Sec : B-integers}

We are interested in the analogue of the set $\N$ in the context of real Cantor bases.

\begin{definition}
\label{Def : Bintegers}
Let $\Beta=(\beta_n)_{n\in \Z}$ be a two-way real Cantor base. A non-negative real number $x$ is called a \emph{$\Beta$-integer} if its $\Beta$-expansion is of the form
\[
	\DB(x)=a_{n-1}\cdots a_0\bigcdot 0^\omega
\]
with $n\in \N$. Otherwise stated, its $\Beta$-expansion has only zeros after the radix point. In this case, we say that $n$ is the \emph{length} of the $\Beta$-expansion of $x$. The set of all $\Beta$-integers is denoted by $\N_{\Beta}$.
\end{definition}

In the case where $\beta_n=\beta$ for all $n\in \N$, then the set of $\Beta$-integers coincides with the classical $\beta$-integers introduced by Gazeau~\cite{Gazeau1997} and widely studied, see for instance~\cite{BurdikFrougnyGazeauKrejcar1998,Fabre1995}.
Moreover, note that $\N_{\Beta}=\N$ if and only if all bases $\beta_n$ are integers.

Clearly, the set of $\Beta$-integers is unbounded.
Let us show that it is also a discrete set.

\begin{proposition}
\label{Pro : noaccumulation}
The set $\N_{\Beta}$ has no accumulation point in $\R$.
\end{proposition}

\begin{proof}
It suffices to see that for all $n\ge 0$, the set of $\Beta$-integers that are smaller than $\beta_{n-1}\cdots \beta_0$ is finite. The $\Beta$-expansion of any such $\Beta$-integer is of the form $a_ma_{m-1}\cdots a_0\bigcdot 0^\omega$ with $m\le n$. Each digit $a_i$ being bounded by $\beta_i$, there are only finitely many $\Beta$-expansions having this property.
\end{proof}

In view of this result, there exists an increasing sequence $(x_k)_{k \in \N}$ such that
\[
	\N_{\Beta} = \{x_k: k \in \N\}.
\]
In particular, we have $x_0=0$ and $x_1=1$, and $\lim_{k \to +\infty }x_k = +\infty$.

\begin{definition}
\label{Def : Mn}
For every $n \in \N$, we define
\[
	M_{\Beta,n} =\max\{x \in \N_{\Beta} : x< \beta_{n-1}\cdots \beta_0 \}
\]
and we write the $\beta$-expansion of $M_n$ as
\[
	\DB(M_{\Beta,n}) = m_{\Beta,n,n-1}\cdots m_{\Beta,n,0}\bigcdot 0^\omega.
\]
\end{definition}

By Lemma~\ref{Lem : RadixOrder}, the number $M_{\Beta,n}$ is the largest $\Beta$-integer among those having a $\Beta$-expansion of length at most $n$.

\begin{proposition}
\label{Pro : Xk+1-Xk}
Let $k\in\N$ and write $\DB(x_k)=\cdots a_2a_1a_0\bigcdot 0^\omega$ and $\DB(x_{k+1})=\cdots b_2b_1b_0\bigcdot 0^\omega$. Define $n$ to be the maximal index such that $a_n\neq b_n$. Then
\begin{enumerate}
\item $x_{k+1}-x_k =\beta_{n-1}\cdots \beta_0 - M_{\Beta,n}\leq 1$
\item $\DB(x_k) = \cdots a_{n+2}a_{n+1}a_{k}m_{\Beta,n,n-1} m_{\Beta,n,n-2} \cdots m_{\Beta,n,0}\bigcdot 0^\omega$
\item $\DB(x_{k+1}) = \cdots b_{n+2}b_{n+1}(a_n+1) 0^n\bigcdot 0^\omega$.
\end{enumerate}
\end{proposition}

\begin{proof}
By definition of $M_{\Beta,n}$, adding one to the right-most digit $m_{\Beta,n,0}$ violates the greedy condition, which means that $M_{\Beta,n}+1\ge \beta_{n-1}\cdots \beta_0$. The rest of the statement follows from the fact that the $\Beta$-expansions of consecutive $\Beta$-integers are consecutive in the radix order.
\end{proof}

The above theorem states that the distances between consecutive $\Beta$-integers take values of the form
\begin{equation}
	\Delta_n:=\beta_{n-1}\cdots\beta_0-M_{\Beta,n},\quad\text{ for }n\in\N,
\end{equation}
accordingly to the position where they differ. Quite obviously, we have $\Delta_0=1-0=1$ and $\Delta_n<1$ for $n\ne 0$.
At specific cases it may happen that $\Delta_n=\Delta_{n'}$ even though $n\neq n'$. Nevertheless, as will be seen later, it is still important to keep track of the position $n$.

In view of the previous proposition, we encode the distances between consecutive $\Beta$-integers by an infinite sequence.

\begin{definition}
\label{Def : Word}
With a two-way real Cantor base $\Beta$, we associate  a sequence $w_{\Beta}=(w_k)_{k\in\N}$ over $\N$ as follows: for all $k\in \N$, we let $w_k$ be the greatest non-negative integer $n$ such that $\DB(x_k)$ and $\DB(x_{k+1})$ differ at index $n$.
\end{definition}

If the set of distances $\{x_{k+1}-x_k : k\in\N\}=\{\Delta_n: n \in \N\}$ between $\Beta$-integers  is finite, then we can project $w_{\Beta}$ onto a sequence over a finite alphabet and study its substitutivity. To see the connection of this notion to the R\'enyi numeration system, we  recall that Fabre~\cite{Fabre1995} assigned to a Parry base $\beta$ a primitive substitution over a finite alphabet. This substitution turned out to be the substitution fixing the $\beta$-integers defined by Gazeau~\cite{Gazeau1997}. Our aim in the next sections will be to find a generalization of these results.

For a two-way real Cantor base $\Beta=(\beta_n)_{n\in Z}$, we write $S(\Beta)$ the two-way real Cantor base $S(\Beta)=(\beta_{n+1})_{n\in \Z}$. Thus, we have
\[
	\val_{\Beta}(a_{N-1}\cdots a_0 \bigcdot a_{-1}a_{-2}\cdots)
	=\beta_0\cdot\val_{S(\Beta)}(a_{N-1}\cdots a_1\bigcdot a_0 a_{-1}a_{-2}\cdots).
\]
For $\N\in\N$, we let $S^N$ be the $N$-th iterate of $S$, i.e., $S^N(\Beta)=(\beta_{n+N})_{n\in \Z}$. We extend this definition to all integers $N\in\Z$ in a natural fashion simply by setting $S^N(\Beta)=(\beta_{n+N})_{n\in \Z}$.

\begin{lemma}
\label{Lem : BandSB}
For all real number $x >0$, we have that $\DB(x)=a_{N-1}\cdots a_0 \bigcdot a_{-1}  a_{-2}\cdots$ if and only if $d_{S(\Beta)}(\frac{x}{\beta_0})=a_{N-1}\cdots a_1\bigcdot a_0 a_{-1}\cdots.$
\end{lemma}

\begin{proof}
This follows from the definition of $S(\Beta)$ and the greedy algorithm.
\end{proof}

We now describe the relationship between $\Beta$-integers and $S(\Beta)$-integers.

\begin{lemma}
\label{Lem : NSB and NB}
We have $\beta_0\N_{S(\Beta)}
	\subset \N_{\Beta}
	\subset \beta_0\N_{S(\Beta)} + \{0,\ldots,\ceil{\beta_0}-1\}$.
\end{lemma}

\begin{proof}
This follows from Lemma~\ref{Lem : BandSB}.
\end{proof}

\begin{lemma}
\label{Lem:dSB-MSB}
For all $n\in\N$, we have
$d_{S(\Beta)}(M_{S(\Beta),n})
=m_{\Beta,n+1,n}\cdots m_{\Beta,n+1,1}\bigcdot 0^\omega$.
\end{lemma}

\begin{proof}
By Lemma~\ref{Lem : BandSB}, we know that
\[
	d_{S(\Beta)}\left(\frac{M_{\Beta,n+1}}{\beta_0}\right)=m_{\Beta,n+1,n}\cdots m_{\Beta,n+1,1}\bigcdot m_{\Beta,n+1,0}0^\omega.
\]
From this, it follows that
\[
	d_{S(\Beta)}\left(\frac{M_{\Beta,n+1}-m_{\Beta,n+1,0}}{\beta_0}\right)=m_{\Beta,n+1,n}\cdots m_{\Beta,n+1,1}\bigcdot 0^\omega.
\]
Now, our aim is to show that
\begin{equation}
\label{eq:M-SB-B}
	M_{S(\Beta),n}=\frac{M_{\Beta,n+1}-m_{\Beta,n+1,0}}{\beta_0}.
\end{equation}
By Lemma~\ref{Lem : RadixOrder}, it then suffices to see that for all $x\in\N_{S(\Beta)}$ such that $x<\beta_n\cdots\beta_1$, the $S(\Beta)$-expansion of $x$ is less than or equal to $m_{\Beta,n+1,n}\cdots m_{\Beta,n+1,1}\bigcdot 0^\omega$ in the radix order. This is indeed the case since if $d_{S(\Beta)}(x)=a_n\cdots a_1\bigcdot 0^\omega$, then $\DB(\beta_0 x)=a_n\cdots a_10\bigcdot 0^\omega$ by Lemma~\ref{Lem : BandSB}. But then by Lemma~\ref{Lem : RadixOrder}, we obtain that $a_n\cdots a_10\bigcdot 0^\omega\le_{\rad} m_{\Beta,n+1,n}\cdots m_{\Beta,n+1,0}\bigcdot 0^\omega$, and hence $a_n\cdots a_1\bigcdot 0^\omega\le_{\rad} m_{\Beta,n+1,n}\cdots m_{\Beta,n+1,1}\bigcdot 0^\omega$.
\end{proof}

We now show that the infinite sequence $w_{S(\Beta)}$ can be mapped to the infinite sequence $w_{\Beta}$ by using a substitution over the infinite alphabet $\N$. This will then allow us to prove that $w_{\Beta}$ is an $S$-adic sequence.

\begin{definition}
We define a map $\psi_{\Beta}$ from $\N$ to $\N^*$ (where $\N$ is seen as an infinite alphabet) as follows:
\[
	\psi_{\Beta}\colon \N \to \N^*,\
	n\mapsto 0^{m_{\Beta,n+1,0}}(n+1).
\]
We see $\psi_{\Beta}$ as a substitution over $\N$, meaning that if $w=(w_k)_{k\in \N}$ is an infinite sequence over $\N$, then $\psi_{\Beta}(w)$ is the infinite concatenation of the blocks $\psi_{\Beta}(w_k)$.
\end{definition}

\begin{proposition}
\label{Pro : Substitution}
We have $\psi_{\Beta}(w_{S(\Beta)})=w_{\Beta}$.
\end{proposition}

\begin{proof}
Let $(x_k)_{k \in \N}$  be the strictly increasing sequence for which $\N_{\Beta} = \{x_k: k \in \N\}$ and let $(\tilde{x}_k)_{k \in \N}$  be the strictly increasing sequence for which $\N_{S(\Beta)} = \{\tilde{x}_k: k \in \N\}$. Write $w_{\Beta}=w_0w_1\cdots$ and $w_{S(\Beta)}=\tilde{w}_0\tilde{w}_1\cdots$ where the $w_k$'s and the $\tilde{w}_k$'s are letters in  $\N$. Let $k\in \N$ and let $n = \tilde{w}_k$. By Proposition~\ref{Pro : Xk+1-Xk}, we have $\tilde{x}_{k+1}-\tilde{x}_k=\beta_n\cdots\beta_1-M_{S(\Beta),n}=:\tilde{\Delta}_n$. By Lemma~\ref{Lem:dSB-MSB}, we also have
\[
	d_{S(\Beta)}(M_{S(\Beta),n})
	=m_{\Beta,n+1,n}\cdots m_{\Beta,n+1,1}\bigcdot 0^\omega.
\]
Proposition~\ref{Pro : Xk+1-Xk} allows us to write
\begin{alignat*}{7}
	d_{S(\Beta)}(\tilde{x}_k)
	&= a_{N-1}\cdots a_{n+1}&a_n \quad\;&\;  m_{\Beta,n+1,n}&\cdots &\; m_{\Beta,n+1,1}\;&\bigcdot \; 0^\omega\\
	d_{S(\Beta)}(\tilde{x}_{k+1})
	&=a_{N-1}\cdots a_{n+1}&(a_n+1) &\qquad 0&\cdots &\qquad 0\;&\bigcdot \; 0^\omega
\end{alignat*}
for some index $N> n$. By Lemma~\ref{Lem : NSB and NB}, both $\beta_0 \tilde{x}_k$ and $\beta_0\tilde{x}_{k+1}$ belong to $\N_{\Beta}$. By Lemma~\ref{Lem : BandSB}, their $\Beta$-expansions are
\begin{alignat*}{7}
	\DB(\beta_0\tilde{x}_k)
	&=a_{N-1}\cdots
	& a_n \quad\;	
	& m_{\Beta,n+1,n}	
	& \cdots
	& m_{\Beta,n+1,1} 	
	& \;0 	
	& \bigcdot 0^\omega\\
	\DB(\beta_0\tilde{x}_{k+1})
	&=a_{N-1}\cdots 	
	&(a_n+1) 		
	&\qquad 0 	
	& \cdots
	& \qquad 0 			
	& 0 	
	& \bigcdot 0^\omega.
\end{alignat*}
Lemma~\ref{Lem : RadixOrder} implies that the $\Beta$-expansion of every $\Beta$-integer laying between  $\beta_0\tilde{x}_k$ and $\beta_0\tilde{x}_{k+1}$ is of the form
\[
	a_{N-1}\cdots  a_n  m_{\Beta,n+1,n}\cdots m_{\Beta,n+1,1}  b\bigcdot  0^\omega
\]
where $b\in \{0,\ldots, m_{\Beta,n+1,0} \}$. We have
\[
	\val_{\Beta}(a_{N-1}\cdots a_n m_{\Beta,n+1,n}\cdots m_{\Beta,n+1,1} b\bigcdot 0^\omega)
	=\beta_0\tilde{x}_k+b.
\]
The letter assigned by Definition~\ref{Def : Word} to the distance between $\beta_0\tilde{x}_k+b$ and $\beta_0\tilde{x}_k+b+1$ is $0$ for each $b \in \{0,\ldots, m_{\Beta,n+1,0}-1 \}$, whereas the letter assigned to the distance between $\beta_0\tilde{x}_k+m_{\Beta,n+1,0}$ and $\beta_0\tilde{x}_{k+1}$ is $n+1$. Since $k$ has been chosen arbitrarily, the sequence $w_{\Beta}$ is the infinite concatenation of the blocks
\[
 	\psi_{\Beta}(\tilde{w}_k)=0^{m_{\Beta,n+1,0}}(n+1)
\]
for each $k$. Otherwise stated, we have $\psi_{\Beta}(w_{S(\Beta)})=w_{\Beta}$, as announced.
\end{proof}

From the proof of the previous proposition, we derive the relation between distances $\Delta_n$ between consecutive $\Beta$-integers and distances $\widetilde{\Delta}_n$ between consecutive $S(\Beta)$-integers.

\begin{corollary}
\label{Lem : relationDelta}
For all $n\in\N$, we have
\[
	\beta_0 \widetilde{\Delta}_n
	=  \Delta_{n+1}+m_{\Beta,n+1,0}.
\]
\end{corollary}

\begin{proof}
Let $n\in\N$. By Proposition~\ref{Pro : Xk+1-Xk} and Definition~\ref{Def : Word}, we have
$\widetilde{\Delta}_n
=\beta_n\cdots\beta_1-M_{S(\Beta),n}$
and
$\Delta_{n+1}
=\beta_n\cdots\beta_0-M_{\Beta,n+1}$.
Then, by the equality~\eqref{eq:M-SB-B} obtained in the proof of Lemma~\ref{Lem:dSB-MSB}, we see that
$\beta_0 \widetilde{\Delta}_n-\Delta_{n+1}
=M_{\Beta,n+1}-\beta_0M_{S(\Beta),n}
=m_{\Beta,n+1,0}$.
\end{proof}

We are now ready to prove that the $\Beta$-integers form an $S$-adic system.

\begin{corollary}
\label{Cor : S-adic}
The infinite sequence $w_{\Beta}$ is the $S$-adic sequence given by the sequence of substitutions $(\psi_{S^n(\Beta)})_{n\in \N}$ applied on the letter $0$, i.e.,
\[
	w_{\Beta}=\lim_{n\to +\infty} \psi_{\Beta} \circ \psi_{S(\Beta)}\circ \cdots \circ \psi_{S^n(\Beta)}(0).
\]
\end{corollary}

\begin{proof}
For all $n\in \Z$, we have $m_{S^n(\Beta),1,0}\ge 1$ since $\beta_n>1$. So, the image of $0$ under $\psi_{S^n(\Beta)}$ has length at least $2$ and starts with at least one zero. Hence, for all $n\in \N$, the (finite) word $\psi_{\Beta} \circ \psi_{S(\Beta)}\circ \cdots \circ \psi_{S^{n-1}(\Beta)}(0)$ is a strict prefix of $\psi_{\Beta} \circ \psi_{S(\Beta)}\circ \cdots \circ \psi_{S^n(\Beta)}(0)$. Therefore, the sequence of finite words $(\psi_{\Beta} \circ \psi_{S^{-1}(\Beta)}\circ \cdots \circ \psi_{S^{-n}(\Beta)}(0))_{n\in\N}$ converges to some limit infinite sequence with respect to the prefix distance. Moreover, it follows from Proposition~\ref{Pro : Substitution} that $\psi_{S^n(\Beta)}(w_{S^{n+1}(\Beta)})=w_{S^n(\Beta)}$ for all $n\in \Z$. This implies that for all $n\in\N$, the word $\psi_{\Beta} \circ \psi_{S(\Beta)}\circ \cdots \circ \psi_{S^n(\Beta)}(0)$ is a prefix of $w_{\Beta}$. Hence the limit is indeed the infinite sequence $w_{\Beta}$.
\end{proof}

\section{Computing the distances between consecutive $\Beta$-integers thanks to the quasi-greedy expansions}
\label{Sec : quasi-greedy}

In the paper~\cite{CharlierCisternino2021}, the admissible Cantor base expansions were characterized. In order to be able to exploit the results from~\cite{CharlierCisternino2021}, we extend the definition of \emph{quasi-greedy expansions of 1} to two-way Cantor real bases. For $\Beta=(\beta_n)_{n\in\Z}$ a two-way Cantor real base, we consider the limit
\[
	\lim_{x\to 1^-}\DB(x).
\]
This sequence has the form $0\, \bigcdot\, a_{-1}a_{-2}a_{-3}\cdots$ where the first digit $a_{-1}$ is positive. Moreover, it does not end with a tail of zeroes and it evaluates to $1$, meaning that $\val_{\Beta}(\DB^*(1))=1$. We then define the \emph{quasi-greedy expansion of $1$} as the infinite word
\[
	\DB^*(1)=a_{-1}a_{-2}a_{-3}\cdots
\]
made of the digits after the radix point. We can then work with the usual lexicographic order between infinite words.

\begin{theorem}\cite[Theorem 26]{CharlierCisternino2021}
\label{CC}
A sequence $0\bigcdot c_{-1}c_{-2} \cdots $ is the $\Beta$-expansion of some number $x \in [0,1)$ if and only if $c_{n-1}c_{n-2}\cdots <_{\lex} d_{S^n(\Beta)}^*(1)$ for all indices $n\le 0$.
\end{theorem}

Note that the previous theorem is also valid for positive indices $n$ since for $n > 0$, the digit $c_{n-1}$ is zero and the leading digit of $d_{S^n(\Beta)}^*(1)$ is positive.

In view of Theorem~\ref{CC}, we see that we need to compare all shifts of a sequence to the quasi-greedy expansions of $1$ with respect to the corresponding shifted bases $S^n(\Beta)$. From now on, for all $n\in \N$, we write
\[
	d_{S^n(\Beta)}^*(1)= d_{n,1}d_{n,2}\cdots.
\]
Observe that unlike previously, we have increasing indices as we read from left to right. With this notation, we can express the distances between consecutive $\Beta$-integers in terms of the digits of the quasi-greedy expansions in the corresponding bases.

\begin{proposition}
\label{Pro : DistancesAndQuasiGreedy}
For all $n\in \N$, we have
\[
	\DB(M_{\Beta,n})=d_{n,1}\cdots d_{n,n}\bigcdot 0^\omega
	\quad \text{ and }\quad
	\Delta_n=\val_{\Beta}(0\bigcdot d_{n,n+1}d_{n,n+2}\cdots).
\]
\end{proposition}

\enlargethispage{2\baselineskip}
\begin{proof}
Let $n\in\N$. As previously, write $\DB(M_{\Beta,n})=m_{\Beta,n,n-1}\cdots m_{\Beta,n,0}\bigcdot 0^\omega$. On the one hand, by Theorem~\ref{CC}, the truncated infinite sequence $0\bigcdot d_{n,1}\cdots d_{n,n} 0^\omega$ is the $S^n(\Beta)$-expansion of some $x\in[0,1)$. By Lemma~\ref{Lem : BandSB}, we get that $\DB(\beta_{n-1}\cdots \beta_0x)= d_{n,1}\cdots d_{n,n}\bigcdot 0^\omega$. By definition of $M_{\Beta,n}$ and by Lemma \ref{Lem : RadixOrder}, we obtain that $d_{n,1}\cdots d_{n,n}\bigcdot 0^\omega\le_{\rad} \DB(M_{\Beta,n})$, and hence that $d_{n,1}\cdots d_{n,n}\le_{\lex} m_{\Beta,n,n-1}\cdots m_{\Beta,n,0}$. On the other hand, Lemma~\ref{Lem : BandSB} also implies that
\[
	d_{S^n(\Beta)}\left(\frac{M_{\Beta,n}}{\beta_{n-1}\cdots \beta_0}\right)=0\bigcdot m_{\Beta,n,n-1}\cdots m_{\Beta,n,0} 0^\omega.
\]
But then Theorem~\ref{CC} gives us that $m_{\Beta,n,n-1}\cdots m_{\Beta,n,0} 0^\omega \le_{\lex} d_{S^n(\Beta)}^*(1)$, and hence we obtain that $m_{\Beta,n,n-1}\cdots m_{\Beta,n,0} \le_{\lex} d_{n,1}\cdots d_{n,n}$. Therefore, we get the first desired equality, i.e., $\DB(M_{\Beta,n})=d_{n,1}\cdots d_{n,n}\bigcdot 0^\omega $.

Next, we compute
\begin{align*}
	\val_{\Beta}(0\bigcdot d_{n,n+1}d_{n,n+2}\cdots)
	&=\val_{\Beta}(d_{n,1}\cdots d_{n,n}\bigcdot d_{n,n+1}d_{n,n+2}\cdots)-\val_{\Beta}(d_{n,1}\cdots d_{n,n}\bigcdot 0^\omega) \\
	&=\beta_{n-1}\cdots\beta_0 \cdot \val_{S^n(\Beta)}(0\bigcdot d_{n,1} d_{n,2}\cdots)-M_{\Beta,n} \\
	&=\beta_{n-1}\cdots\beta_0-M_{\Beta,n} \\
	&=\Delta_n
\end{align*} 
which gives us the second equality of the statement. 
\end{proof}

\begin{corollary}
\label{Cor:limit}
For all $n\in\N$, we have $
	0\bigcdot d_{n,n+1}d_{n,n+2}\cdots
	=\lim_{x\to (\Delta_n)^-} d_{\Beta}(x)$.
\end{corollary}

\begin{proof}
This is a consequence of Lemma~\ref{Lem : BandSB} and Proposition~\ref{Pro : DistancesAndQuasiGreedy}.
\end{proof}

\section{$\Beta$-integers  in  systems  associated with alternate bases}
\label{Sec : U-integers in alternate bases}

A two-way Cantor real base $\Beta=(\beta_n)_{n\in \Z}$ is called an \emph{alternate base} if  it is periodic, that is, if there exists  a positive integer $p$ such that $\beta_{n+p}=\beta_n$ for all $n\in \Z$. In this case, we simply denote $\Beta=(\beta_{p-1},\ldots,\beta_0)$ and the integer $p$ is called the \emph{period} of $\Beta$.

\begin{definition}
An alternate base $\Beta$ is called \emph{Parry} if  the sequence  $d_{S^i(\Beta)}^*(1)$ is eventually periodic for every  $i \in\{0,\ldots, p-1\}$.
\end{definition}

The above notion allows us to characterize alternate bases  $\Beta$ for which the set of $\Beta$-integers has a finite number of distances between consecutive elements.

\begin{proposition}
\label{Prop:Parry-distances}
Let $\Beta=(\beta_{p-1},\ldots,\beta_0)$ be an alternate base. Then the set of distances between consecutive $\Beta$-integers is finite if and only if the base $\Beta$ is Parry.
\end{proposition}

\begin{proof}
By Proposition~\ref{Pro : DistancesAndQuasiGreedy}, we have to show that the set
\[
	D=\{\val_{\Beta}(0\bigcdot d_{n,n+1}d_{n,n+2}\cdots) : n \in \N\}
\]
is finite if and only if $\Beta$ is Parry. For $i\in\{0,\ldots, p-1\}$ we let
\[
	D_i=\{\val_{\Beta}(0\bigcdot d_{i,n+1}d_{i,n+2}\cdots) : n \in \N,\ n\equiv i \bmod p\}.
\]
Since $D=\cup_{i=0}^{p-1}D_i$, the set $D$ is finite if and only if $D_i$ is finite for every $i$. We will show that for every $i$, the sequence $d_{S^ i(\Beta)}^*(1)$ is eventually periodic if and only if $D_i$ is finite.
 Let us fix  $i \in \{0,\ldots,p-1\}$.

First, suppose that $d_{S^ i(\Beta)}^*(1)$ is eventually periodic. Without loss of generality we may assume that there exist $\ell\ge 0$ and $m \ge 1$ such that $d_{S^i(\Beta)}^*(1)$ has preperiod $\ell p$ and period $mp$. Then for all $n>\ell p$ such that $n\equiv i \bmod p$, we have $d_{i,n+1}d_{i,n+2}\cdots
= d_{i,n+mp+1}d_{i,n+mp+2}\cdots$.
Hence, the set $D_i$ has at most $\ell + m$ elements.

Conversely, suppose that the set $D_i$ is finite.  Then there exist  $n,n' \in \N$ such that $n<n'$,   $n\equiv i \bmod p$, $n'\equiv i \bmod p$ and $\Delta_n=\val_{\Beta}(0\bigcdot d_{i,n+1}d_{i,n+2}\cdots)= \val_{\Beta}(0\bigcdot d_{i,n'+1}d_{i,n'+2}\cdots)=\Delta_{n'}$.
Let us show that if the valuation of these two strings in base $\Beta$ coincide, then the strings must coincide as well.
By Corollary~\ref{Cor:limit} and since $n\equiv i \bmod p$ and $n'\equiv i \bmod p$, we have
\[
	0\bigcdot d_{i,n+1}d_{i,n+2}\cdots
	= \lim_{x\to (\Delta_n)^-} \DB(x)
	= \lim_{x\to (\Delta_{n'})^-} \DB(x)
	= 0\bigcdot d_{i,n'+1}d_{i,n'+2}\cdots.
\]
Consequently, the sequence $d_{i,n}d_{i,n+1}\cdots $ is purely periodic and thus $d_{S^ i(\Beta)}^*(1) =d_{i,1}d_{i,1}\cdots$ is eventually periodic.
\end{proof}

For an alternate base $\Beta=(\beta_{p-1},\dots,\beta_0)$, since $\Beta$ and  $S^p(\Beta)$ coincide, the sequence  $w_{\Beta}$ is fixed by the composition $\psi_{\Beta}\circ \cdots \circ \psi_{S^{p-1}(\Beta)}$. In the case where $\Beta$ is Parry, the sequence of distances between consecutive elements in $\N_{\Beta}$ can be coded by an infinite word $v_{\Beta}$ over a finite alphabet.

\begin{definition}
\label{Def : v-phi}
Let $\Beta=(\beta_{p-1},\dots,\beta_0)$ be a Parry alternate base, and let $\ell\ge 0$ and $m\ge 1$ such that $d_{S^i(\Beta)}(1)$ has preperiod $\ell p$ and period $mp$ for every $i\in\{0,\ldots, p-1\}$. Let then $\A$ be the finite alphabet $\{0,\ldots,\ell p+mp-1\}$. We define an infinite word $v_{\Beta}=(v_k)_{k\in\N}$ over $\A$ by setting
\[
 	v_k =
 	\begin{cases}
  	 	w_k, 	& \text{if } 0\le w_k < \ell p;\\
   		\ell p + ((w_k-\ell p) \bmod mp), 	& \text{otherwise}
	\end{cases}
\]
(where the letters $w_k$ are those from Definition~\ref{Def : Word}). We also define a substitution $\varphi_{\Beta}\colon\A^*\to \A^*$ by setting
\[
	\varphi_{\Beta}(n)
	=\begin{cases}
		0^{d_{n+1,n+1}}(n+1),
			&\text{if } 0\le n< \ell p +mp-1; \\
		0^{d_{0,\ell p+mp}}(\ell p) ,
			& \text{if } n=\ell p +mp-1
 \end{cases}
\]
for each $n\in\A$.
\end{definition}

\begin{proposition}
Let $\Beta=(\beta_{p-1},\cdots,\beta_0)$ be a Parry  alternate base, let $v_{\Beta}=(v_k)_{k\in\N}$ as in Definition~\ref{Def : v-phi} and let $(x_k)_{k\in\N}$ be the increasing sequence of $\Beta$-integers. If $k,k'\in \N$ are such that $v_k=v_{k'}$ then $x_{k+1} - x_k = x_{k'+1}-x_{k'}$.
\end{proposition}

\begin{proof}
Proposition \ref{Pro : DistancesAndQuasiGreedy} and the periodicity of  $d_{S^i(\Beta)}^*(1)$ imply that $\Delta_n=	\Delta_{n+mp}$ for all $n\ge \ell p$. Then, by Proposition~\ref{Pro : Xk+1-Xk}, for all $n\ge \ell p$ and $N\ge 0$, distances of consecutive $\Beta$-integers coded in $w_{\Beta}$ by the letters $n$ and $n+Nmp$ coincide. Since the word $v_{\Beta}$ is obtained from $w_{\Beta}$ by identifying the letters of the form $n+Nmp$ for $n\ge \ell p$ and $N \ge 0$ with the letter $n$, we get that the infinite word $v_{\Beta}$ satisfies the property of the statement, that is, if $v_k=v_{k'}$ then $x_{k+1} - x_k = x_{k'+1}-x_{k'}$.
\end{proof}

We show that the composition $\varphi_{\Beta}\circ \varphi_{S(\Beta)}\circ \cdots \circ \varphi_{S^{p-1}(\Beta)}$ gives us a primitive substitution that fixes the infinite word $v_{\Beta}$.

\begin{theorem}
\label{Thm : FinitSubstit}
Let $\Beta=(\beta_{p-1},\cdots,\beta_0)$ be a Parry alternate base. The substitution $\varphi_{\Beta}\circ \varphi_{S(\Beta)}\circ \cdots \circ \varphi_{S^{p-1}(\Beta)}$ is primitive and $\varphi_{\Beta}\circ \varphi_{S(\Beta)}\circ \cdots \circ \varphi_{S^{p-1}(\Beta)} (v_{\Beta}) = v_{\Beta}$.
\end{theorem}

\begin{proof}
By Proposition~\ref{Pro : Substitution}, the substitution $\psi_{\Beta}$ maps $w_{S(\Beta)}$ to $w_{\Beta}$. By Proposition~\ref{Pro : DistancesAndQuasiGreedy}, we have $m_{\Beta,n+1,0}=d_{n+1,n+1}$ for all $n$. So we can rewrite the images of the letters $n$ under $\psi_{\Beta}$ as $\psi_{\Beta}(n)=0^{d_{n+1,n+1}}(n+1)$. Let us show that the substitution $\varphi_{\Beta}$ maps $v_{S(\Beta)}$ to $v_{\Beta}$. In our setting, we have $d_{\ell p + mp,\ell p + mp} = d_{0,\ell p + mp}$ and $d_{n+1+Nmp,n+1+Nmp}= d_{n+1,n+1}$ for all $n\ge \ell p$ and $N \ge 0$. We get that $\psi_{\Beta}(n + Nmp)=0^{d_{n+1,n+1}}(n+1+Nmp)$ for all $n\ge \ell p$ and $N \ge 0$. Since the word $v_{\Beta}$ is obtained from $w_{\Beta}$ by identifying the letters of the form $n+Nmp$ for $n\ge \ell p$ and $N \ge 0$ with the letter $n$ and the same holds true for $v_{S(\Beta)}$ and $w_{S(\Beta)}$, the substitution $\varphi_{\Beta}$ maps $v_{S(\Beta)}$ to $v_{\Beta}$. If we now consider the substitutions $\varphi_{S^n(\Beta)}$, which are defined as in Definition~\ref{Def : v-phi} but for the shifted bases $S^n(\Beta)$ instead of $\Beta$ itself, we can deduce that $\varphi_{S^n(\Beta)}(v_{S^{n+1}(\Beta)})=v_{S^n(\Beta)}$ for all $n$. Now, since $S^p(\Beta) = \Beta$, the word $v_{\Beta}$ is fixed by $\varphi = \varphi_{\Beta}\circ \varphi_{S(\Beta)}\circ \cdots \circ \varphi_{S^{p-1}(\Beta)}$.

It remains to show that the substitution $\varphi_{\Beta}\circ \varphi_{S(\Beta)}\circ \cdots \circ \varphi_{S^{p-1}(\Beta)}$ is primitive.
In~\cite{CharlierCisternino2021}, a deterministic finite automaton associated with a Parry alternate base was defined. We will exploit this automaton by using the special form of its adjacency matrix, which can be interpreted in terms of the incidence matrices of the $p$ substitutions $\varphi_{S^i(\Beta)}$ for $i\in\{0,\ldots,p-1\}$.

We depict only the graph underlying this automaton, as specifying the initial and final states is irrelevant here. The set of vertices is $Q=\{0,\ldots,p-1\}\times \{0,\ldots,\ell p+mp-1\}$. In what follows, operations on the first component of a vertex is always considers modulo $p$. For each vertex $(i,n)$ with $n<\ell p+mp-1$, there is an arrow labeled by $d_{i+n+1,n+1}$ from $(i,n)$ to $(i-1,n+1)$. Moreover, there is an arrow from each vertex $(i,\ell p+mp-1)$ to $(i-1,\ell p)$ which is labeled by $d_{i,\ell p+mp}$. Finally, for each vertex $(i,n)$ and $s\in\{0,\ldots,d_{i+n+1,n+1}-1\}$, there is an arrow labeled by $s$ from $(i,n)$ to $(i-1,0)$. This construction is illustrated in Section~\ref{Sec : Example} on an example and the corresponding graph is depicted in Figure~\ref{Fig : Automaton}.

Let us recall that the adjacency matrix of a graph with $d$ vertices is a square matrix $M \in \N^{d\times d}$ such that for each pair $a,b$ of vertices of the graph, the entry $M_{a,b}$ is equal to the number of arrows starting in $b$ and ending in $a$. We will later use the well known fact that $(M^j)_{a,b}$ equals the number of oriented  paths of length $j$ starting in the vertex $b$ and ending in the vertex $a$. Similarly, the incidence matrix of a substitution $\sigma$ over a finite alphabet $\A$ of size $d$ is a square matrix $M \in \N^{d\times d}$ such that for each pair $a,b$ of letters in $\A$, the entry $M_{a,b}$ is equal to the number of occurrences of the letter $a$ in the image $\sigma(b)$.

By construction, the adjacency matrix of our graph where the vertices are ordered as
\[
	(p-1,0),\ldots,(p-1,\ell p+mp-1),
	\ldots,
	(0,0),\ldots,(0,\ell p+mp-1),\]
is of the form
\[
M = \left( \begin{array}{cccccc}
		\Theta & M_{p-1} & \Theta & \cdots & \Theta \\
		\Theta & \Theta & M_{p-2} & \cdots & \Theta \\
 		\vdots &&&& \vdots \\
 		\Theta & \Theta & \Theta & \cdots &  M_1\\
		M_0 & \Theta & \Theta & \cdots & \Theta
	\end{array}\right),
\]
where $\Theta$ is the zero matrix of size $\ell p+mp$ and each block $M_i$ is a square matrix of size $\ell p+mp$.

We claim that each $M_i$ is the transposed incidence matrix of the substitution $\varphi_{S^i(\Beta)}$. Fix some $i\in\{0,\ldots,p-1\}$ and consider $r,s\in\{0,\ldots,\ell p+mp-1\}$. Then $(M_i)_{r,s}$ is the number of arrows from the vertex $(i,r)$ to the vertex $(i-1,s)$. By definition of the graph, this number is equal to $1$ either if $r<\ell p+mp-1$ and $s=r+1$, or if $r=\ell p+mp-1$ and $s=\ell p$, it is equal to $d_{i+r+1,r+1}$ if $s=0$, and it is equal to $0$ otherwise. Now, consider the incidence matrix $N_i$ of $\varphi_{S^i(\Beta)}$. We want to show that $(N_i)_{s,r}=(M_i)_{r,s}$. The entry $(N_i)_{s,r}$ is the number of occurrences of the letter $s$ in the image $\varphi_{S^i(\Beta)}(r)$. By definition of the substitution $\varphi_{S^i(\Beta)}$, there are  $d_{i+r+1,r+1}$ occurrences of the letter $0$ and one occurrence of the letter $r+1$ in $\varphi_{S^i(\Beta)}(r)$ if $r<\ell p+mp-1$, while there are $d_{i,\ell p+mp}$ occurrences of the letter $0$ and one occurrence of the letter $\ell p$ in $\varphi_{S^i(\Beta)}(\ell p+mp-1)$. Since $d_{i,\ell p+mp}=d_{i+\ell p+mp,\ell p+mp}$, we obtain the desired claim.

Consequently, the $p$-th power of $M$ has the form
\begin{equation}
\label{powerP}
M^p = \left( \begin{array}{cccccc}
		D_{p-1} & \Theta & \cdots & \Theta & \Theta \\
		\Theta & D_{p-2} & \cdots & \Theta & \Theta \\
 		\vdots &&&& \vdots \\
 		\Theta & \Theta & \cdots & D_{p-2} & \Theta \\
		\Theta & \Theta & \cdots & \Theta & D_0
	\end{array}\right),
\end{equation}
where $D_{p-1} = M_{p-1}\cdots M_1 M_0$ is the transposed incidence matrix of the substitution $\varphi_{\Beta}\circ \varphi_{S(\Beta)}\circ \cdots \circ \varphi_{S^{p-1}(\Beta)}$.

We now show that our graph is strongly connected, i.e., that there exists a cycle visiting all vertices of the underlying graph. We consider three kinds of subpaths. First, for every $i\in\{0,\ldots,p-1\}$, the prefix $d_{i+1,1}\cdots d_{i+1,\ell p+mp}$ of length $\ell p+mp$ of $d_{S^i(\Beta)}^*(1)$ labels a path going through the vertices
\begin{align*}
	&(i,0),\ldots,(i-p+1,p-1), \\
	&(i,p),\ldots,(i-p+1,2p-1), \\
	&\ldots, \\
	&(i,\ell p),\ldots,(i-p+1,(\ell+1)p-1), \\
	&\ldots, \\
	&(i,(\ell+m-1)p),\ldots,(i-p+1,(\ell+m)p-1), \\
	&(i,\ell p).
\end{align*}
Second, for every $i\in\{0,\ldots,p-1\}$, the periodic part $d_{i+1,\ell p+1}\cdots d_{i+1,\ell p+mp}$ of $d_{S^i(\Beta)}^*(1)$  must contain a non-zero digit $d_{i+1,j_i}$, where $j_i\in\{\ell p+1,\ldots,\ell p+mp\}$. Therefore, there is an arrow labeled by $0$ from the vertex $(i-j_i+1,j_i-1)$ to $(i-j_i,0)$. Note that the vertex $(i-j_i+1,j_i-1)$ belongs to the cycle
\begin{align*}
	&(i,\ell p),\ldots,(i-p+1,(\ell+1)p-1), \\
	&\ldots, \\
	&(i,(\ell+m-1)p),\ldots,(i-p+1,(\ell+m)p-1),\\
	&(i,\ell p)
\end{align*}
considered before. Third, we show that there is a cycle visiting all the vertices of the form $(i,0)$ for $i\in\{0,\ldots,p-1\}$. Every quasi-greedy-expansions $d_{S^i(\Beta)}^*(1)$ starts with a non-zero digit $d_{i,1}$. Therefore, for every $i\in\{0,\ldots,p-1\}$, there is an arrow labeled by $0$ from $(i,0)$ to $(i-1,0)$. Thus, there exist exponents $t_i\ge 1$ such that the word $0^{t_i}$ labels a path from the vertex $(i-j_i+1,j_i-1)$ to the vertex $(i-1,0)$. Altogether, we see that the word
\begin{align*}
	& (d_{0,1}\cdots d_{0,\ell p+mp})
	 (d_{0,\ell p+1}\cdots d_{0,j_0-1})
	 0^{t_{p-1}} \\
	& (d_{p-1,1}\cdots d_{p-1,\ell p+mp})
	(d_{p-1,\ell p+1}\cdots d_{p-1,j_{p-1}-1})
	0^{t_{p-2}} \\
	&\cdots \\
	& (d_{1,1}\cdots d_{1,\ell p+mp})
	(d_{1,\ell p+1}\cdots d_{1,j_1}-1)
	0^{t_0}
\end{align*}
labels a cycle from $(p-1,0)$ to $(p-1,0)$ that visits all vertices of the graph.

By definition of the graph, for any given $i\in\{0,\ldots,p-1\}$, the length of any path from a vertex $(i,n)$ to a vertex $(i,n')$ must be a multiple of $p$. We deduce that the matrices $D_i$ are irreducible. Moreover, the first entry on the diagonal of $D_i$ is positive since the word $0^p$ labels a path from $(i,0)$ to $(i,0)$. These two properties force primitivity of every matrix $D_i$, and hence of the substitution $\varphi_{\Beta}\circ \varphi_{S(\Beta)}\circ \cdots \circ \varphi_{S^{p-1}(\Beta)}$.
\end{proof}

We conclude this section by proving an important consequence of Theorem~\ref{Thm : FinitSubstit} about the uniqueness of the alternate base. We first prove a lemma which is of interest by itself.

\begin{lemma}
\label{Lem:PF-eigenvalue}
Let $\Beta =(\beta_{p-1},\ldots,\beta_0)$ be a Parry alternate base. The Perron–Frobenius eigenvalue of the incidence matrix of the substitution $\varphi_{\Beta}\circ \varphi_{S(\Beta)}\circ \cdots \circ \varphi_{S^{p-1}(\Beta)}$ is given by the product $\beta_{p-1}\cdots \beta_0$ of the bases.
\end{lemma}

\begin{proof}
Let us keep the notation from the proof of Theorem~\ref{Thm : FinitSubstit}. The incidence matrix of the substitution $\varphi_{S^i(\Beta)}$ is given by
\[
	(M_i)^\intercal = \begin{pmatrix}
		d_{i+1,i+1}	& d_{i+2,i+2} & \cdots & d_{i+\ell p+mp-1,i+\ell p+mp-1}  & d_{i,i+i\ell p+mp}\\
		1 & 0	& \cdots & 0 &0 \\
		0 & 1	& \cdots & 0 &0 \\
		\vdots &&&&\\
		0 & 0	& \cdots & 1 &0
		\end{pmatrix} +E_i	
\]
where $E_i$ is the square matrix of size $\ell p+mp$ defined by $E_{\ell p-1,\ell p+mp-1}=1$ and $E_{n,n'}=0$ otherwise.

Now, for each $i\in\{0,\ldots,p-1\}$, let $\Delta_{i,n}$ denote the distance between $S^i(B)$-integers differing at index $n$. We thus have $\Delta_{i,n}=\beta_{i+n-1}\cdots\beta_i-M_{S^i(B),n}$. From Corollary~\ref{Lem : relationDelta} combined with Proposition~\ref{Pro : DistancesAndQuasiGreedy}, we know that $\beta_i \Delta_{i+1,n} =  \Delta_{i,n+1}+d_{i+n+1,i+n+1}$ for all $i$ and $n$. We focus on the values $n\in\{0,\ldots,\ell p+mp-1\}$ in order to write, for every $i\in\{0,\ldots,p-1\}$, the corresponding $\ell p + mp$ equalities in a matrix form. In order to do so, we define $\Delta^{(i)} =(\Delta_{i,0}, \ldots, \Delta_{i,\ell p+mp -1})$. Since $\Delta_{i,0}=1$ and $\Delta_{i,\ell p+mp}= \Delta_{i,\ell p }$ for all $i\in\{0,\ldots,p-1\}$, we get the matrix equalities
\[
	\beta_i\Delta^{(i+1)}  =\Delta^{(i)} (M_i)^\intercal,\quad i\in\{0,1,\ldots,p-1\}
\]

From this, we derive
\begin{align*}
	\beta_0\cdots \beta_{p-2}\beta_{p-1}	\Delta^{(0)}
	&=\beta_0\cdots \beta_{p-2}\beta_{p-1}	\Delta^{(p)} \\
	&=\Delta^{(0)}(M_0)^\intercal\cdots (M_{p-2})^\intercal(M_{p-1})^\intercal \\
	&=\Delta^{(0)}(M_{p-1}M_{p-2}\cdots M_0)^\intercal \\
	&=\Delta^{(0)} D_{p-1}.
\end{align*}
In other words, the primitive matrix $D_{p-1}$ has a positive left eigenvector $\Delta^{(0)}$ corresponding to the eigenvalue $\beta_{p-1}\beta_{p-2}\cdots \beta_0$. The Perron-Frobenius theorem implies that this eigenvalue is precisely the Perron-Frobenius eigenvalue.
\end{proof}

\begin{theorem}
Let $\Beta =(\beta_{p-1},\cdots,\beta_0)$ be a Parry alternate base. Then no other alternate base of length $p$ has the same list of quasi-greedy expansions of $1$.
\end{theorem}

\begin{proof}
Assume that the list of quasi-greedy expansions $d_{S^i(\Beta)}^*(1)$ for $i\in\{0,\ldots, p-1\}$ also gives the list of quasi-greedy expansions of $1$ associated with an alternate base $\Gamma = (\gamma_{p-1}, \gamma_{p-2}, \ldots, \gamma_0)$ which is different from $\Beta$. But for each $i\in\{0,\ldots,p-1\}$, the substitution $\varphi_{S^{i}(\Beta)}$ only depends on the digits of $d_{S^i(\Beta)}^*(1)$. Therefore, we must have $\varphi_{S^{i}(\Gamma)}=\varphi_{S^{i}(\Beta)}$ for every $i\in\{0,\ldots,p-1\}$, and thus also $\varphi_{\Gamma}\circ \varphi_{S(\Gamma)}\circ \cdots \circ \varphi_{S^{p-1}(\Gamma)}=\varphi_{\Beta}\circ \varphi_{S(\Beta)}\circ \cdots \circ \varphi_{S^{p-1}(\Beta)}$. By Lemma~\ref{Lem:PF-eigenvalue}, we obtain that the products $\beta_{p-1}\cdots \beta_1\beta_0$ and $\gamma_{p-1}\cdots \gamma_1\gamma_0$ are the same.
Then by \cite[Proposition 20]{Charlier&Cisternino&Masakova&Pelantova:2022}, we derive that $\Gamma=B$.
\end{proof}

Let us mention that the question of the uniqueness of an alternate base given a list of quasi-greedy expansions of 1 remains open in case that the base is not assumed to be a Parry alternate base.

\section{A running example}
\label{Sec : Example}

In this section, we illustrate our constructions and results on an example. Consider the alternate base $\Beta=(\frac{1+ \sqrt{13}}{2},\frac{5+ \sqrt{13}}{6})$. It is Parry since $\DB^*(1)=20(01)^\omega$ and $d_{S(B)}^*(1)=(10)^\omega$. We set $\ell=m=2$, that is we consider the writings $20(01)^\omega$ and $10(10)^\omega$ of the two quasi-greedy expansions of $1$ in order to have common preperiods and periods which are multiple of $p=2$.

By Proposition~\ref{Pro : DistancesAndQuasiGreedy}, we can compute the largest $\Beta$-integers $M_{\Beta,n}$ with a $\Beta$-expansion of length less than or equal to $n$. The numbers $M_{\Beta,n}$ and (the left part of) their $\Beta$-expansions are given in Table~\ref{Table:M} for $n\le 7$.
\begin{table}[htb]
\[
\begin{array}{c|r|r}
n & \DB(M_{\Beta,n}) & M_{\Beta,n} \\
\hline
0 & \varepsilon 	& 0\\
1 & 1 			& 1 \\
2 & 20 			& \frac{5+ \sqrt{13}}{3} \\
3 & 101 			& \frac{5+ \sqrt{13}}{2} \\
4 & 2001 		&  \frac{17+ 4\sqrt{13}}{3} \\
5 & 10101 		& 8+ 2\sqrt{13} \\
6 & 200101 		& \frac{109+ 29\sqrt{13}}{6} \\
7 & 1010101 		& 26+ 7\sqrt{13} \\
\end{array}
\]
\caption{The first numbers $M_{\Beta,n}$ for the alternate base $\Beta=(\frac{1+ \sqrt{13}}{2},\frac{5+ \sqrt{13}}{6})$.}
\label{Table:M}
\end{table}
Then, in Table~\ref{Table:B-integers-wB-vB} are given the $\Beta$-integers with a $\Beta$-expansions of length at most $5$ (up to two decimal digits), their $\Beta$-expansions, the corresponding prefixes of the sequence $w_{\Beta}$ and its projection $v_{\Beta}$ onto the finite alphabet $\{0,\ldots,\ell p+mp-1\}=\{0,1,2,3\}$.
\begin{table}[htb]
\[
\begin{array}{c|c|r|c|c||c|c|r|c|c||c|c|r|c|c}
k & x_k & \DB(x_k) & w_{\Beta} & v_{\Beta} &
k & x_k & \DB(x_k) & w_{\Beta} & v_{\Beta}
&
k & x_k & \DB(x_k) & w_{\Beta} & v_{\Beta}\\
\hline
0 & 0	& \varepsilon 	& 0 & 0 &
12 & 8.03	& 1100			& 0 & 0 &
24 & 16.64	& 100001			& 1 & 1\\
1 & 1 	& 1 				& 1 & 1 &
13 & 9.03 	& 1101 			& 3 & 3 &
25 & 17.07	& 100010			& 0 & 0\\
2 & 1.43 	& 10 			& 0 & 0 &
14 & 9.47 	& 2000			& 0 & 0 &
26 & 18.07 	& 100011			& 1 & 1\\
3 & 2.43 	& 11 			& 1 & 1 &
15 & 10.47	& 2001	 		& 4 & 2 &
27 & 18.51	& 100020			& 2 & 2\\
4 & 2.86 	& 20 			& 2 & 2 &
16 & 10.90	& 10000 			& 0 & 0 &
28 & 18.94	& 100100			& 0 & 0\\
5 & 3.30 	& 100 			& 0 & 0 &
17 & 11.90	& 10001 			& 1 & 1 &
29 & 19.94	& 100101			& 3 & 3\\
6 & 4.30 	& 101 			& 3 & 3 &
18 & 12.34	& 10010 			& 0 & 0 &
30 & 20.38	& 101000			& 0 & 0\\
7 & 4.73 	& 1000			& 0 & 0 &
19 & 13.34 	& 10011 			& 1 & 1 &
31 & 21.38	& 101001			& 1 & 1\\
8 & 5.73 	& 1001 			& 1 & 1 &
20 & 13.77	& 10020 			& 2 & 2 &
32 & 21.81	& 101010			& 0 & 0\\
9 & 6.17 	& 1010			& 0 & 0 &
21 & 14.21	& 10100 			& 0 & 0 &
33 & 22.81	& 101011			& 1 & 1\\
10 & 7.17	& 1011			& 1 & 1 &
22 & 15.21	& 10101 			& 5 & 3 &
34 & 23.25	& 101020			& 2 & 2\\
11 & 7.60	& 1020 			& 2 & 2 &
23 & 15.64	& 100000 		& 0 & 0 &
35 & 23.68	& 101100			& 0 & 0
\end{array}
\]
\bigskip

\caption{$\Beta$-expansions of the first $\Beta$-integers and the corresponding prefixes of
$w_{\Beta}$ and $v_{\Beta}$, with respect to the alternate base $\Beta=(\frac{1+ \sqrt{13}}{2},\frac{5+ \sqrt{13}}{6})$.}
\label{Table:B-integers-wB-vB}
\end{table}
From the proof of Proposition~\ref{Prop:Parry-distances}, we see that there can be at most $\ell p+mp=4$ possible distances between consecutive $\Beta$-integers. In our case, we only get two of them since
\begin{align*}
\Delta_0		&=\val_{\Beta}(0\bigcdot d_{0,1}d_{0,1}\cdots)
			=\val_{\Beta}(0\bigcdot 20(01)^\omega)
			=1 \\
\Delta_1		&=\val_{\Beta}(0\bigcdot d_{1,2}d_{1,3}\cdots)
			=\val_{\Beta}(0\bigcdot (01)^\omega)
			=\frac{-1+\sqrt{13}}{6}
			\sim 0.434259 \\
\Delta_2		&=\val_{\Beta}(0\bigcdot d_{0,3}d_{0,4}\cdots)
			=\val_{\Beta}(0\bigcdot (01)^\omega)
			=\Delta_1 \\
\Delta_3		&=\val_{\Beta}(0\bigcdot d_{1,4}d_{1,5}\cdots)
			=\val_{\Beta}(0\bigcdot (01)^\omega)
			=\Delta_1.
\end{align*}
The first $\Beta$-integers, the $\Beta$-expansions of which are given in Table~\ref{Table:B-integers-wB-vB}, are represented in Figure~\ref{Fig:beta-integers}.
\begin{figure}[htb]
\includegraphics[scale=1]{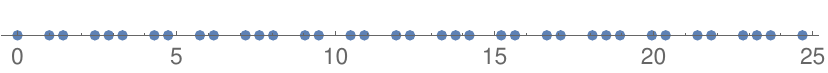}
\caption{The first $\Beta$-integers, with respect to the alternate base $\Beta=(\frac{1+ \sqrt{13}}{2},\frac{5+ \sqrt{13}}{6})$.}
\label{Fig:beta-integers}
\end{figure}

Now, we consider the shifted base $S(\Beta)=(\frac{5+ \sqrt{13}}{6},\frac{1+ \sqrt{13}}{2})$, which we simply denote by $\tilde{\Beta}$ for conciseness (this brings no ambiguity since $p=2$).
To illustrate Lemma~\ref{Lem : NSB and NB}, the $\Beta$-integers belonging to $\beta_0\N_{\tilde{\Beta}}$ are those having a right-most digit $0$ in their $\Beta$-expansion. In Figure~\ref{Fig:beta-integers-2} are represented the $\Beta$-integers from Figure~\ref{Fig:beta-integers} that belong to $\beta_0\N_{\tilde{\Beta}}$.
\begin{figure}[htb]
\includegraphics[scale=1]{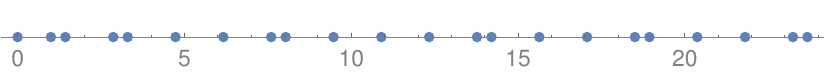}
\caption{The first $\Beta$-integers belonging to the subset $\beta_0\N_{\tilde{\Beta}}$, with respect to the alternate base $\Beta=(\frac{1+ \sqrt{13}}{2},\frac{5+ \sqrt{13}}{6})$.}
\label{Fig:beta-integers-2}
\end{figure}
On the other hand, since $\lceil\beta_0\rceil=\lceil\frac{5+\sqrt{13}}{6}\rceil=2$, the second inclusion from Lemma~\ref{Lem : NSB and NB} gives $\N_{\Beta}\subset\beta_0\N_{\tilde{\Beta}}+\{0,1\}$, which can be seen in Table~\ref{Table:B-integers-wB-vB} as all $\Beta$-integers end either in a digit $0$ or $1$. Note that some numbers in $\beta_0\N_{\tilde{\Beta}}+\{0,1\}$ are not $\Beta$-integers. For example, $x_4=\frac{5+ \sqrt{13}}{3}$ belongs to $\beta_0\N_{\tilde{\Beta}}$ as $d_{\Beta}(x_4)=20\bigcdot 0^\omega$, but $x_4+1=\frac{8+ \sqrt{13}}{3}$ is not a $\Beta$-integer since $\DB(x_4+1)=100\bigcdot 101 0^\omega$.

Let us denote the $k$-th $\tilde{\Beta}$-integer by  $\tilde{x}_k$ and the distance between consecutive $\tilde{\Beta}$-integers differing in (maximal) position $n$ by
$\tilde{\Delta}_n$. The $\tilde{\Beta}$-integers $M_{\tilde{\Beta},n}$ and their $\Beta$-expansions are given in Table~\ref{Table:SB-M} for $n\le 7$.
\begin{table}[htb]
\[
\begin{array}{c|r|r}
n & d_{\tilde{\Beta}}(M_{\tilde{\Beta},n}) & M_{\tilde{\Beta},n} \\
\hline
0 & \varepsilon 	& 0\\
1 & 2 			& 2 \\
2 & 10 			& \frac{1+ \sqrt{13}}{3} \\
3 & 200 			& 3+ \sqrt{13} \\
4 & 1010 		&  \frac{9+ 3\sqrt{13}}{2} \\
5 & 20010 		& \frac{23+ 7\sqrt{13}}{2} \\
6 & 101010 		& 17+ 5\sqrt{13} \\
7 & 2001010 		& \frac{81+ 23\sqrt{13}}{2} 
\end{array}
\]
\caption{The first numbers $M_{\tilde{\Beta},n}$ for the alternate base $\tilde{\Beta}=(\frac{5+ \sqrt{13}}{6},\frac{1+ \sqrt{13}}{2})$.}
\label{Table:SB-M}
\end{table}
Table~\ref{Table:SB-integers-wSB-vSB} is the analogue of Table~\ref{Table:B-integers-wB-vB} for the shifted base $\tilde{\Beta}$.
\begin{table}[htb]
\[
\begin{array}{c|c|r|c|c||c|c|r|c|c||c|c|r|c|c}
k & \tilde{x}_k & d_{\tilde{\Beta}}(\tilde{x}_k) & w_{\tilde{\Beta}} & v_{\Beta} &
k & \tilde{x}_k & d_{\tilde{\Beta}}(\tilde{x}_k) & w_{\tilde{\Beta}} & v_{\tilde{\Beta}}
&
k & \tilde{x}_k & d_{\tilde{\Beta}}(\tilde{x}_k) & w_{\tilde{\Beta}} & v_{\tilde{\Beta}}\\
\hline
0 	& 0		& \varepsilon 	& 0 & 0 &
12 	& 9.90	& 1010			& 4 & 2 &
24 	& 20.51	& 11002			& 1 & 1\\
1 	& 1 		& 1 				& 0 & 0 &
13 	& 10.90 	& 10000 			& 0 & 0 &
25 	& 20.81	& 11010			& 4 & 2\\
2 	& 2 	& 2	 			& 1 & 1 &
14 	& 11.90 	& 10001			& 0 & 0 &
26 	& 21.81 	& 20000			& 0 & 0\\
3 	& 2.30 	& 10 			& 2 & 2 &
15 	& 12.90	& 10002	 		& 4 & 2 &
27 	& 22.81	& 20001			& 0 & 0\\
4 	& 3.30 	& 100 			& 0 & 0 &
16 	& 13.21	& 10010 			& 2 & 2 &
28 	& 23.81	& 20002			& 1 & 1\\
5 	& 4.30 	& 101 			& 0 & 0 &
17 	& 14.21	& 10100 			& 0 & 0 &
29 	& 24.11	& 20010			& 5 & 2\\
6 	& 5.30 	& 102 			& 1 & 1 &
18 	& 12.21	& 10101 			& 0 & 0 &
30 	& 25.11	& 100000			& 0 & 0\\
7 	& 5.60 	& 110			& 2 & 2 &
19 	& 16.21 	& 10102 			& 1 & 1 &
31 	& 26.11	& 100001			& 0 & 0\\
8 	& 6.60 	& 200 			& 3 & 3 &
20 	& 16.51	& 10110 			& 2 & 2 &
32 	& 27.11	& 100002			& 0 & 0\\
9 	& 7.60 	& 1000			& 0 & 0 &
21 	& 17.51	& 10200 			& 3 & 3 &
33 	& 27.42	& 100010			& 2 & 2\\
10 	& 8.60	& 1001			& 0 & 0 &
22 	& 18.51	& 11000 			& 0 & 0 &
34 	& 28.42	& 100100			& 0 & 0\\
11 	& 9.60	& 1002 			& 1 & 1 &
23 	& 19.51	& 11001 			& 0 & 0 &
35 	& 29.42	& 100101			& 0 & 0
\end{array}
\]
\bigskip

\caption{$\tilde{\Beta}$-expansions of the first $\tilde{\Beta}$-integers and the corresponding prefixes of
$w_{\tilde{\Beta}}$ and $v_{\tilde{\Beta}}$, with respect to the alternate base $\tilde{\Beta}=(\frac{5+ \sqrt{13}}{6},\frac{1+ \sqrt{13}}{2})$.}
\label{Table:SB-integers-wSB-vSB}
\end{table}
For the base $\tilde{\Beta}$, we get two possible distances between consecutive $\tilde{\Beta}$-integers:
\begin{align*}
\tilde{\Delta}_0		&=\val_{\tilde{\Beta}}(0\bigcdot d_{1,1}d_{1,1}\cdots)
			=\val_{\tilde{\Beta}}(0\bigcdot (10)^\omega)
			=1 \\
\tilde{\Delta}_1		&=\val_{\tilde{\Beta}}(0\bigcdot d_{0,2}d_{0,3}\cdots)
			=\val_{\tilde{\Beta}}(0\bigcdot 0(01)^\omega)
			=\frac{-3+\sqrt{13}}{2}
			\sim 0.302776 \\
\tilde{\Delta}_2		&=\val_{\tilde{\Beta}}(0\bigcdot d_{1,3}d_{1,4}\cdots)
			=\val_{\tilde{\Beta}}(0\bigcdot (10)^\omega)
			=\tilde{\Delta}_0	 \\
\tilde{\Delta}_3		&=\val_{\tilde{\Beta}}(0\bigcdot d_{0,4}d_{0,5}\cdots)
			=\val_{\Beta}(0\bigcdot (10)^\omega)
			=\tilde{\Delta}_0	.
\end{align*}
The first $\tilde{\Beta}$-integers, the $\tilde{\Beta}$-expansions of which are given in Table~\ref{Table:B-integers-wB-vB}, are represented in Figure~\ref{Fig:beta-integers-S}.
\begin{figure}[htb]
\includegraphics[scale=1]{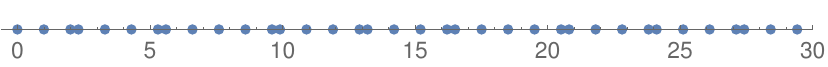}
\caption{The first $\Beta$-integers, with respect to the alternate base $\tilde{\Beta}=(\frac{5+ \sqrt{13}}{6},\frac{1+ \sqrt{13}}{2})$.}
\label{Fig:beta-integers-S}
\end{figure}

\enlargethispage{2\baselineskip}
We now illustrate Proposition~\ref{Pro : Substitution}, which says that $\psi_{\Beta}(w_{\tilde{B}})=w_{\Beta}$. From Table~\ref{Table:M}, we see that the map $\psi_{\Beta}\colon\N\to\N^*$ is given by $0\mapsto 01,\ 1\mapsto 2$ and $n\mapsto 0(n+1)$ for all $n\ge 2$.
Computing a prefix
\begin{align*}
	\psi_{\Beta}(w_{\tilde{B}})
	&=\psi_{\Beta}(001 2 001 2 3 001 4 001 2
	001 2 3 001 4 001 5 \cdots) \\
&=01012 03 01012 03 04 01012 05 01012
03 01012 03 04 01012 05 01012 06 \cdots
\end{align*}
of $\psi_{\Beta}(w_{\tilde{B}})$, we see that it indeed coincides with a prefix of $w_{\Beta}$.

Now, let us compute the graph described in the proof of Theorem~\ref{Thm : FinitSubstit} associated with the base $\Beta$. This graph is depicted in Figure~\ref{Fig : Automaton}.
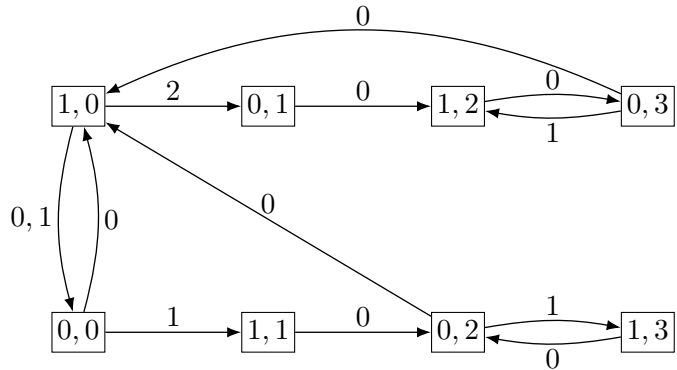
\begin{figure}[htb]
\centering
\begin{tikzpicture}
\tikzstyle{every node}=[shape=rectangle, fill=none, draw=black,
minimum size=15pt, inner sep=2pt]
\node (1) at (0,3) {$1,0$};
\node (2) at (2.5,3) {$0,1$};
\node (3) at (5,3) {$1,2$};
\node (4) at (7.5,3) {$0,3$};

\node (6) at (0,0) {$0,0$};
\node (7) at (2.5,0) {$1,1$};
\node (8) at (5,0) {$0,2$};
\node (9) at (7.5,0) {$1,3$};

\tikzstyle{every node}=[shape=rectangle, fill=none, draw=black, minimum size=15pt, inner sep=2pt]
\tikzstyle{every path}=[color=black, line width=0.5 pt]
\tikzstyle{every node}=[shape=rectangle, minimum size=5pt, inner sep=2pt]

\draw [-Latex] (1) to node [above] {$2$} (2);
\draw [-Latex] (1) to [bend right=15] node [left] {$0,1$} (6);

\draw [-Latex] (2) to node [above] {$0$} (3);

\draw [-Latex] (3) to [bend left=10] node [above] {$0$} (4);
\draw [-Latex] (4) to [bend left=10] node [below] {$1$} (3);

\draw [-Latex] (4) to [bend right=25] node [above] {$0$} (1);

\draw [-Latex] (6) to node [above] {$1$} (7);
\draw [-Latex] (6) to [bend right=15] node [right] {$0$} (1);

\draw [-Latex] (7) to node [above] {$0$} (8);
\draw [-Latex] (8) to [bend left=10] node [above] {$1$} (9);
\draw [-Latex] (8) to node [above] {$0$} (1);
\draw [-Latex] (9) to [bend left=10] node [below] {$0$} (8);

\end{tikzpicture}
\caption{The graph associated with the base $\Beta=(\frac{1+ \sqrt{13}}{2},\frac{5+ \sqrt{13}}{6})$.}
\label{Fig : Automaton}
\end{figure}
The set of vertices is given by $\{0,1\}\times \{0,1,2,3\}$. Following the prefix $2001$ of $\DB^*(1)$ starting from the vertex $(1,0)$, we successively visit the states $(0,1)$, $(1,2)$, $(0,3)$ and then $(1,2)$ again. Similarly, by following the prefix $1010$ of $d_{S(B)}^*(1)$ starting from the vertex $(0,0)$, we successively visit the states $(1,1)$, $(0,2)$, $(1,3)$ and then $(0,2)$ again. When we read a letter less than the one prescribed by the current quasi-greedy expansion from a vertex of the form $(i,n)$, we go to the vertex $(i-1,0)$. We thus get the arrows from $(1,0)$ to $(0,0)$ labeled by $0$ and $1$, the arrow from $(0,3)$ to $(1,0)$ labeled by $0$, and the arrow from $(0,2)$ to $(1,0)$ labeled by $0$. Note that the obtained graph is bipartite: we always go from states with first component $0$ to states with first component $1$, and conversely, we always go from states with first component $1$ to states with first component $0$. So, if we order the states as
\[
	(1,0),(1,1),(1,2),(1,3),(0,0),(0,1),(0,2),(0,3),
\]
the corresponding adjacency matrix if given by
\[
	M=\begin{pmatrix}
	0 & 0 & 0 & 0 & 2 & 1 & 0 & 0 \\
	0 & 0 & 0 & 0 & 0 & 0 & 1 & 0 \\
	0 & 0 & 0 & 0 & 0 & 0 & 0 & 1 \\
	0 & 0 & 0 & 0 & 0 & 0 & 1 & 0 \\
	1 & 1 & 0 & 0 & 0 & 0 & 0 & 0 \\
	0 & 0 & 1 & 0 & 0 & 0 & 0 & 0 \\
	1 & 0 & 0 & 1 & 0 & 0 & 0 & 0 \\
	1 & 0 & 1 & 0 & 0 & 0 & 0 & 0
	\end{pmatrix}
\]
The matrices
\[
	M_0=\begin{pmatrix}
	1 & 1 & 0 & 0 \\
	0 & 0 & 1 & 0 \\
	1 & 0 & 0 & 1 \\
	1 & 0 & 1 & 0
	\end{pmatrix}
	\quad \text{ and }\quad
	M_1=\begin{pmatrix}
	2 & 1 & 0 & 0 \\
	0 & 0 & 1 & 0 \\
	0 & 0 & 0 & 1 \\
	0 & 0 & 1 & 0
	\end{pmatrix}
\]
are respectively the transposed incidence matrices of the substitutions $\varphi_{\Beta}$ and $\varphi_{\tilde{\Beta}}$ over the alphabet $\{0,1,2,3\}$, which are defined by
\[
	\varphi_{\Beta}\colon
	\begin{cases}
	0\mapsto 01\\
	1\mapsto 2\\
	2\mapsto 03\\
	3\mapsto 02
	\end{cases}
	\quad \text{ and }\quad
	\varphi_{\tilde{\Beta}}\colon
	\begin{cases}
	0\mapsto 001\\
	1\mapsto 2\\
	2\mapsto 3\\
	3\mapsto 2.
	\end{cases}
\]
Observe that, as described in the proof of Theorem~\ref{Thm : FinitSubstit}, the associated graph is strongly connected and the word $(2001)(0)(00)(1010)(0)$ labels a path from $(1,0)$ to $(1,0)$. The matrix
\[
	D_1=M_1M_0=\begin{pmatrix}
	2 & 2 & 1 & 0 \\
	1 & 0 & 0 & 1 \\
	1 & 0 & 1 & 0 \\
	1 & 0 & 0 & 1
	\end{pmatrix}
\]
is thus the transposed incidence matrix of the substitution
\[
	\varphi_{\Beta}\circ\varphi_{\tilde{\Beta}}\colon
	\begin{cases}
	0\mapsto 01012\\
	1\mapsto 03\\
	2\mapsto 02\\
	3\mapsto 03.\\
	\end{cases}
\]
This matrix is primitive as we have
\[
	(D_1)^3=\begin{pmatrix}
	23 & 14 & 10 & 6 \\
	10 & 6 & 4 & 3 \\
	10 & 6 & 5 & 2 \\
	10 & 6 & 4 & 3
	\end{pmatrix}.
\]
Computing a prefix
\[
	01012\cdot 03\cdot 01012\cdot 03\cdot 02\cdot 01012\cdot 03\cdot 01012 \cdot 03\cdot 01012\cdot 03\cdot 02\cdots
\]
of the fixed point of the substitution $\varphi_{\Beta}\circ\varphi_{\tilde{\Beta}}$, we see that it indeed corresponds to the prefix of the infinite word $v_{\Beta}$ coding the distances between $\Beta$-integers given in Table~\ref{Table:B-integers-wB-vB}.

Moreover, the characteristic polynomial of $D_1$ is given by $X^4-4X^3+2X^2+X$, which factors as $X(X-1)(X^2-3X-1)$. The eigenvalues of $D_1$ are thus given by $0,1,\frac{3+\sqrt{13}}{2},\frac{3-\sqrt{13}}{2}$. The dominant eigenvalue is $\frac{3+\sqrt{13}}{2}$, which is precisely the product $\beta_1\beta_0$ as stated in Lemma~\ref{Lem:PF-eigenvalue}.

Note that, in our specific example, since we have $\Delta_1=\Delta_2=\Delta_3$, the image
\[
	\pi(v_{\Beta})=0101101010110101 \cdots
\]
under the projection $\pi\colon\{0,1,2,3\}^*\to\{0,1\}^*,\ 0\mapsto 0,\ 1\mapsto 1,\ 2\mapsto 1,\ 3\mapsto 1$ contains enough information to code the distances between consecutive $\Beta$-integers. Clearly, this new infinite sequence $\pi(v_{\Beta})$ is the fixed point of the substitution given by
\begin{equation}\label{eq:substitutionrunning}
	\begin{cases}	
	0\mapsto 01011\\
	1\mapsto 01.	
	\end{cases}
\end{equation}

\section{\textbf{Sturmian sequences}}
\label{sec:sturmian}

 Balanced sequences over an alphabet $\{a,b\}$
that are not eventually periodic are called \emph{sturmian}. Recall that a binary sequence
$(v_n )_{n\in \N}$ is \emph{balanced} if for every length $n\in \N$ and every $i,j \in \N$ the number of occurrences of the letter $a$ in the words  $v_{i}v_{i+1}\cdots v_{i+n-1}$ and $v_{j}v_{j+1}\cdots v_{j+n-1}$ differs at most by $1$.

There are many equivalent definitions of sturmian sequences, see~\cite{Lothaire2002}. For our purposes,
we use the result of~\cite{Berstel&Seebold:1994} which characterizes sturmian sequences among fixed points of primitive substitutions: An infinite binary word over the alphabet $\{0,1\}$ that is fixed by a primitive morphism is sturmian if and only if the morphism belongs to the so-called sturmian monoid, i.e., the monoid of binary morphisms generated by the following morphisms:
\[
	E\colon\begin{cases}	
	0\mapsto 1\\
	1\mapsto 0
	\end{cases}
	\qquad
	G\colon\begin{cases}	
	0\mapsto 0\\
	1\mapsto 01
	\end{cases}\qquad \text{and}
	\qquad
	\tilde{G}\colon\begin{cases}	
	0\mapsto 0\\
	1\mapsto 10
	\end{cases}.
\]
Sturmian morphisms are precisely those which map a sturmian sequence to a sturmian sequence. For more details on sturmian morphisms, see \cite{Lothaire2002}.

Sequences $v_B$, as introduced in Definition~\ref{Def : v-phi}, are fixed by a substitution over an alphabet with $\ell p +mp$  letters where $\ell$  and $m$ are chosen so that all quasi-greedy expansions of 1 have preperiods and periods of the same lengths $\ell p$ and $mp$ respectively.

As we have seen on the example in Section~\ref{Sec : Example}, in some cases this alphabet is exaggerated and the sequence of distances between consecutive $\Beta$-integers could be coded by a projection  $\pi$ of $v_B$ onto a smaller alphabet. The new sequence $\pi(v_B)$ can also be fixed by a substitution.

In this section, we are interested in the cases where $v_{\Beta}$ itself (not its projection) is a sturmian sequence. This situation happens to be quite rare, even if we allow alternate base numeration systems. The following proposition describes all the possible cases.

\begin{proposition}
\label{prop:sturmian}
Let $\Beta =(\beta_{p-1},\ldots,\beta_0)$ be a Parry alternate base. The infinite sequence  $v_{\Beta}$ is sturmian if and only if one of the following cases is satisfied.
\begin{itemize}
\setlength{\itemindent}{0.5cm}
\item[Case 1.] We have $p=1$ and $\DB^*(1)=(d0)^\omega$ with $d\geq 1$, in which case $\beta_0$ is the positive root of $X^2-dX-1$ and $v_{\Beta}$ is the sturmian sequence fixed by the substitution $0\mapsto 0^d1,\, 1\mapsto 0$.
\item[Case 2.] We have $p=1$ and $\DB^*(1)=(d+1)d^\omega$ for $d\geq 1$, in which case $\beta_0$ is the largest root of $X^2-(d+2)X+1$ and $v_{\Beta}$ is the sturmian sequence fixed by the substitution $0\mapsto 0^{d+1}1,\, 1\mapsto 0^d1$.
\item[Case 3.] We have $p=2$ and $\DB^*(1)=(d0)^\omega$, $d_{S(\Beta)}^*(1)=(e0)^\omega$ with $d,e\geq 1$, in which case $\beta_0$ and $\beta_1$ are the largest roots of $dX^2-deX-e$ and $eX^2-deX-d$ respectively, the product $\beta_1\beta_0$ is the largest root of $X^2-(de+2)X+1$, and $v_{\Beta}$ is the  sturmian sequence fixed by the substitution $0\mapsto(0^e1)^d0,\, 1\mapsto 0^e1$.
\end{itemize}
\end{proposition}

\begin{proof}
The case when $p=1$ has been treated in~\cite{Frougny&Masakova&Pelantova:2004} and \cite{Frougny&Masakova&Pelantova:2007}. Assume that $p\geq 2$ and that $v_{\Beta}$ is sturmian. Since sturmian sequences are binary words and the infinite word $v_{\Beta}$ is defined over the alphabet $\{0,\ldots,\ell p+mp-1\}$, we must have $(\ell+m)p=2$. Since in addition $p\ge 2$ and $m\ge 1$, this gives $p=2$, $\ell=0$ and $m=1$, i.e., $\DB^*(1)$ and $d_{S(\Beta)}^*(1)$ are of the form
  \begin{equation}
  \label{eq:sturm}
  \DB^*(1)=(d_1d_2)^\omega
  \quad \text{and}\quad
  d_{S(\Beta)}^*(1)=(e_1e_2)^\omega.
  \end{equation}
By Theorem~\ref{Thm : FinitSubstit}, the infinite word $v_{\Beta}$ is fixed by the primitive substitution $\varphi_{B}\circ \varphi_{S(B)}$ where
\begin{alignat*}{5}
&  \varphi_{\Beta} \colon && 0\mapsto 0^{e_1}1 	&& \quad\text{and}\quad && \varphi_{S(\Beta)}\colon && 0\mapsto 0^{d_1}1 \\
& 				 && 1\mapsto 0^{d_2+1} 	&& 						&& 	 			&& 1\mapsto 0^{e_2+1}.
\end{alignat*}
Since $v_{\Beta}$ is a sturmian sequence, any substitution that fixes it must be a sturmian morphism. Consequently, its incidence matrix $A_1 A_0$ with
  \[
  A_0=\begin{pmatrix}
  	e_1 & d_2+1 \\
  	1 	& 0
  \end{pmatrix}
  \quad\text{and}\quad
  A_1=\begin{pmatrix}
    d_1 & e_2+1 \\
    1 	& 0
   \end{pmatrix},
  \]
has determinant equal to $\pm 1$, and hence $\det A_0=\pm 1$ and $\det A_1=\pm 1$. This implies $d_2=e_2=0$. Setting $d_1=d$, $e_1=e$ into~\eqref{eq:sturm}, we derive that $1=\frac{d}{\beta_1}+\frac{1}{\beta_1\beta_0}=\frac{e}{\beta_0}+\frac{1}{\beta_1\beta_0}$. It is then not difficult to see that $\beta_0$ and $\beta_1$ are the largest roots of the polynomials $dX^2-deX-e$ and $eX^2-deX-d$, respectively. We then also get that
\begin{align*}
(\beta_1\beta_0)^2
&=d\beta_1\beta_0^2+\beta_1\beta_0\\
&=\beta_1(de\beta_0+e)+\beta_1\beta_0\\
&=(de+1)\beta_1\beta_0+e\beta_1 \\
&=(de+1)\beta_1\beta_0+\beta_0\beta_1-1\\
&=(de+2)\beta_1\beta_0-1
\end{align*}
hence $\beta_1\beta_0$ is the largest root of the polynomial $X^2-(de+2)X+1$ (the other one belonging to the interval $(0,1)$).
Moreover, the corresponding substitution $\varphi_{B}\circ \varphi_{S(B)}$ is given by
\begin{align*}
	\varphi_{B}\circ \varphi_{S(B)}\colon
	& 0\mapsto (0^e1)^d0\\
	& 1\mapsto 0^e1.
\end{align*}

Conversely, suppose that we are in Case 3. Then the morphisms $\varphi_{\Beta}$ and $\varphi_{S(\Beta)}$ are given by
\begin{alignat*}{5}
&  \varphi_{\Beta} \colon && 0\mapsto 0^e1 	&& \quad\text{and}\quad && \varphi_{S(\Beta)}\colon && 0\mapsto 0^d1 \\
& 				 && 1\mapsto 0 	&& 						&& 	 			&& 1\mapsto 0,
\end{alignat*}
which are both sturmian morphisms, as $\varphi_{\Beta}=G^e\circ E$ and $\varphi_{S(\Beta)}=G^d\circ E$. Thus, so is their composition $\varphi_{\Beta}\circ \varphi_{S(\Beta)}$, and hence its fixed point $v_{\Beta}$ is a  sturmian sequence. 
 \end{proof}

In all cases, one can derive frequencies $\rho_0,\rho_1$ of letters 0 and 1 in the sturmian sequence $v_{\Beta}$ from the corresponding substitution. Frequencies  of letters will be expressed in terms of its continued fraction. Recall, that a continued fraction  $[a_0,a_1, a_2, \ldots ] $  with $a_0 \in \mathbb{Z}$  and $a_n \in \mathbb{N}, a_n \geq 1$ for every $n \in \N, n>0$,  represents the irrational number $\gamma$, when
$$
\gamma = \lim_{n\to +\infty} \ a_0 +\cfrac{1}{a_1 +\cfrac{1}{a_2 + \cfrac{1}{\ddots   \ {a_{n-1}+\cfrac{1}{a_n}}}}}.
$$
If the  sequence $a_0, a_1, a_2 , \ldots $ of the so-called partial quotients in the continued fraction is eventually periodic  with the period  $a_{i+1},a_{i+2}, \ldots, a_{i+k}$ we use the notation
$$
[a_0, a_1, \ldots, a_{i} ,\overline{a_{i+1},a_{i+2}, \ldots, a_{i+k}}].
$$

 \begin{corollary}
Let $\Beta =(\beta_{p-1},\ldots,\beta_0)$ be a Parry alternate base such that the infinite sequence $v_{\Beta}$ is sturmian. Then the frequency vector $(\rho_0,\rho_1)$ of the letters $0$ and $1$ in  $v_{\Beta}$ is given as follows, according to the cases described in Proposition~\ref{prop:sturmian}.
\begin{itemize}
\setlength{\itemindent}{0.5cm}
\item[Case 1.] We have $(\rho_0,\rho_1)=(\frac{\beta_0}{\beta_0+1},\frac{1}{\beta_0+1})$ and the continued fraction of $\rho_0$ is $[0,1,\overline{d}]$.
\item[Case 2.] We have $(\rho_0,\rho_1)=(\frac{\beta_0-1}{\beta_0},\frac{1}{\beta_0})$ and the continued fraction of $\rho_0$ is $[0,\overline{1,d}]$.
\item[Case 3.] We have $(\rho_0,\rho_1)=(\frac{\beta_1}{\beta_1+1},\frac{1}{\beta_1+1})$ and the continued fraction of $\rho_0$ is equal to $[0,1,\overline{e,d}]$.
\end{itemize}
 \end{corollary}

\begin{proof}
The case when $p=1$ has been treated in~\cite{Frougny&Masakova&Pelantova:2004} and \cite{Frougny&Masakova&Pelantova:2007}.
Suppose here that $p\ge 2$. The vector $(\rho_0,\rho_1)$ is an eigenvector of the incidence matrix
\[
	\begin{pmatrix}
       de+1 	& e \\
       d		& 1
	\end{pmatrix}
\]
of the substitution $\varphi_{B}\circ \varphi_{S(B)}$ corresponding to the Peron-Frobenius eigenvalue $\beta_1\beta_0$, and of course, such that $\rho_0+\rho_1=1$. So we easily compute that $\rho_0=\frac{\beta_1\beta_0-1}{\beta_1\beta_0-1+d}$. Since we are in Case 3, we know that $1=\frac{d}{\beta_1}+\frac{1}{\beta_1\beta_0}=\frac{e}{\beta_0}+\frac{1}{\beta_1\beta_0}$, hence $\beta_1\beta_0-1=d\beta_0=e\beta_1$. We obtain that $(\rho_0,\rho_1)=(\frac{\beta_0}{\beta_0+1},\frac{1}{\beta_0+1})$.

Let us show that the continued fraction of $\rho_0$ is equal to $[0,1,\overline{e,d}]$. Since $\beta_0$ is a root of $dx^2 - dex-e$, we have $\beta_0 - e - \tfrac{e}{d \beta_0} = 0$. As $d\beta_0 = e\beta_1$, we derive
$\beta_0=e+\frac{1}{\beta_1}$. By similar argumentation, we obtain $\beta_1=d+\frac{1}{\beta_0}$. Together
\[
	\beta_0 = e+ \frac{1}{d+\tfrac{1}{\beta_0}},
\]
and thus $\beta_0$ has the purely periodic continued fraction
$\beta_1 = [\overline{e,d}]$
and for the frequency $\rho_0$ we have
\[
\rho_0 = \frac{\beta_0}{\beta_0+1} = \frac{1}{1+\frac{1}{\beta_0}} = [0,1,\overline{e,d}].
\]

\end{proof}

Surprisingly, we see that one can obtain a sturmian word $v_{\Beta}$ with frequency $\rho_0=[0,\overline{1,a}]$ as a coding of integers in two different numeration systems. For $p=1$, this is only possible for the real bases $\tau$ and $\tau^2$ where $\tau=\frac{1+\sqrt{5}}{2}$ is the golden ratio. On the one hand, the real base $\tau$ fulfills the condition of Case~1 with $d=1$, and we get $\rho_0=\frac{\tau}{\tau+1}=[0,\overline{1}]$. On the other hand, the real base $\tau^2$ fulfills the condition of Case~2 with $d=1$, and we get $\rho_0=\frac{\tau^2-1}{\tau^2}=\frac{\tau}{\tau+1}=[0,\overline{1}]$.

But if we allow alternate bases, there are in fact infinitely many pairs of numeration systems giving the same frequency $\rho_0=[0,\overline{1,a}]$. It happens for $p=1$ with $\DB^*(1)=(a+1)a^\omega$. In this case $\beta_0$ is a root of $X^2-(a+2)X+1$ and the distances between consecutive $\Beta$-integers take values $\Delta_0=1$ and $\Delta_1=\val_{\Beta}(0\bigcdot a^\omega)=\frac{a}{\beta_0-1}$. For example, for $a=2$, we obtain the real base $2+\sqrt{3}$ and the frequency $\rho_0$ is given by $\frac{\beta_0-1}{\beta_0}=-1+\sqrt{3}=[0,\overline{1,2}]$. The infinite word $v_{\Beta}$ is fixed by the substitution $0\mapsto 0001$  and $1\mapsto 001$.

Another possibility is to take $p=2$ with $\DB^*(1)=(10)^\omega$ and $d_{S(\Beta)}^*(1)=(a0)^\omega$.
Here, $\beta_0$ is a root of $X^2-aX-a$ and $\beta_1$ is a root of $aX^2-aX-1$, their product $\beta_1\beta_0$ being a root of $X^2-(a+2)X+1$. The distances between consecutive $\Beta$-integers take values $\Delta_0=1$ and $\Delta_1=\val_{\Beta}(0\bigcdot (0a)^\omega)=\frac{a}{\beta_1\beta_0-1}$. For $a=2$, we get the alternate base $B=(\beta_1,\beta_0)=(\frac{1+\sqrt{3}}{2},1+\sqrt{3})$ and the frequency $\rho_0$ is also given  by $\frac{\beta_0}{\beta_0+1}=-1+\sqrt{3}=[0,\overline{1,2}]$. The infinite word $v_{\Beta}$ is fixed by another substitution, namely,  $0\mapsto 0010$  and $1\mapsto 001$.

From the substitutions we see that although the two sturmian words have the same frequency of letters, the sequences do not coincide: The one from Case 2 with $p=1$ has the prefix $0001$,  whereas the sequence from Case 3 with $p=2$ has the prefix $0010$.

For the list $d_{\Beta}^*(1),  d_{S(\Beta)}^*(1),\ldots,   d_{S^{p-1}(\Beta)}^*(1)$ in general, the associated word $v_\Beta$ from Definition~\ref{Def : v-phi} is rarely binary.  Nevertheless, the phenomenon that gaps between consecutive $\Beta$-integers take only two values can occur for arbitrary $p$ when the quasi-greedy expansions of unity are of specific form. Then the structure of the set of $\Beta$-integers can be coded by a projection $\pi$ of the sequence $v_\Beta$ to a binary alphabet.

In the running example of Section~\ref{Sec : Example}, the projected word $\pi(v_\Beta)$ onto the binary alphabet is fixed by the substitution $\varphi$ defined  in \eqref{eq:substitutionrunning}.
It can be easily verified that $\varphi = E\tilde{G}EG^2E$.
Therefore the infinite word $\pi(v_\Beta)$ is sturmian.
 Noteworthy, we can calculate that the frequencies of letters in this case are $\rho_0=[0,2, \overline{3}]$  and $\rho_1=[0,1,1, \overline{3}]$, and so this sturmian sequence is different from all the cases described in Proposition \ref{prop:sturmian}. The following example illustrates that not every set of $\Beta$-integers that can be coded by a binary infinite word gives a sturmian sequence.

\begin{example}\label{ex:non-sturm}
Consider $p=2$ and $d_{\Beta}^*(1) = (3020)^\omega$ and $  d_{S(\Beta)}^*(1) = 4(2030)^\omega$. For $\beta_0, \beta_1$  and their product $\delta$ we have
 $$
 \begin{aligned}
 1&= \frac{3}{\beta_1} + \frac{2}{\delta\beta_1} + \frac{1}{\delta^2}\\
 1&= \frac{4}{\beta_0} + \frac{2}{\delta} + \frac{3}{\delta^2} +  \frac{2}{\delta^3} +  \frac{1}{\beta_0\delta^3},
 \end{aligned}
 $$
which yields that $\delta $ is a root of $X^2 - 14 X -12$, i.e., $\delta = 7+\sqrt{61}$. Moreover,  $\beta_0 = \frac{\delta^2-1}{3\delta +2} = \frac{11+\sqrt{61}}{4}$
 and $\beta_1 = \frac{\delta}{\beta_0} = \frac{4}{15}(4+ \sqrt{61})$.

According to Proposition \ref{Pro : DistancesAndQuasiGreedy}, there are only two possible distances  between consecutive $\Beta$-integers, which are given by $\Delta_0=\val_{\Beta}(3020)^\omega = 1$ and  $\Delta_1=\val_{\Beta}(2030)^\omega = \frac{2\delta+3}{3\delta +2} = \frac{\sqrt{61}-5}{4}\sim 0.7$. More precisely, for all $n\ge 0$ and $j\in\{0,1,2,3\}$, we have
\[
	\Delta_{4n+j}=\begin{cases}
	\Delta_0, & \text{if }j\in\{0,3\}; \\
	\Delta_1, & \text{if }j\in\{1,2\}.
\end{cases}
\]
By Proposition \ref{Pro : Xk+1-Xk}, we can compute the first $\Beta$-integers; see Table~\ref{Table:not-sturmian}. We see that the binary infinite word $v_{\Beta}$ coding $\Beta$-integers thus contains the word $00$ (coding distances between $x_0, x_1, x_2$) and the word  $11$ (coding the distances between $x_{14}, x_{15}, x_{16}$). Therefore the sequence is not balanced, hence not sturmian.
\begin{table}[htb]
\[
\begin{array}{c|c|r|c|c||c|c|r|c|c}
k & x_k & \DB(x_k) & w_{\Beta} & v_{\Beta} &
k & x_k & \DB(x_k) & w_{\Beta} & v_{\Beta} \\
\hline
0 & 0	& \varepsilon 	& 0 & 0 &
9 & 8.7 & 14	& 1 & 1 \\
1 & 1 	& 1 	& 0 & 0 &
10 & 9.4 	& 20 & 0 & 0 \\
2 & 2 	& 2 	& 0 & 0 &
11 & 10.4 	& 21	& 0 & 0 \\
3 & 3 	& 3 	& 0 & 0 &
12 & 11.4	& 22	 & 0 & 0 \\
4 & 4 & 4 	& 1 & 1 &
13 & 12.4	& 23	 & 0 & 0 \\
5 & 4.7 	& 10 & 0 & 0 &
14 & 13.4	& 24 			& 1 & 1 \\
6 & 5.7 	& 11 & 0 & 0 &
15 & 14.1	& 30 & 2 & 1 \\
7 & 6.7 	& 12 & 0 & 0 &
16 & 14.8 	& 100 & 0 & 0 \\
8 & 7.7 	& 13 & 0 & 0 &
17 & 15.8	& 101 			& 0 & 0
\end{array}
\]
\bigskip

\caption{The first $\Beta$-integers and the corresponding prefixes of
$w_{\Beta}$ and $v_{\Beta}$, with respect to the alternate base $\Beta=(\frac{4}{15}(4+ \sqrt{61}),\frac{11+ \sqrt{61}}{4})$.}
\label{Table:not-sturmian}
\end{table}
\end{example}

\section{Comments and open problems}

The infinite symbolic sequences associated with $\beta$-integers in the context of Rényi numeration systems may have affine factor complexity. The characterization of such cases is given in~\cite{BernatMasakovaPelantova2007} using a detailed description of special factors. This allows one to determine which Parry sequences  are sturmian, or more generally Arnoux-Rauzy.  It turns out that Arnoux-Rauzy sequences originated from $\beta$-integers are standard.  It is worth mentioning that the sturmian sequence coding distances between consecutive $\Beta$-integers of our running example is  not standard.

Infinite symbolic sequences $v_\Beta$ coding $\Beta$-integers as defined in this paper represent a very rich family of infinite sequences. Similarly as we have identified sturmian sequences among the infinite words $v_\Beta$, it would be interesting to find which of those infinite words $v_\Beta$ belong to a recently introduced class of dendric sequences~\cite{BertheDeFeliceDolceLeroyPerrinReutenauerRindone2015}. Dendric sequences appear to be $S$-adic~\cite{BertheDolceDurandLeroyPerrin2018} and generalize two very extensively studied classes of infinite words, namely sequences coding interval exchange transformations and episturmian sequences.


\section{\textbf{Acknowledgment}}
Émilie Charlier is supported by the FNRS grant J.0034.22.
Célia Cisternino is supported by the FNRS grant 1.A.564.19F.
Zuzana Mas\'akov\'a and Edita Pelantov\'a are supported by the European Regional Development Fund project CZ.02.1.01/0.0/0.0/16\_019/0000778.

\bibliographystyle{abbrv}
\bibliography{BetaIntegersFinal}

\end{document}